\documentclass[11pt]{amsart}
\usepackage[utf8x]{inputenc}
\usepackage[english]{babel}
\usepackage[T1]{fontenc} 
\usepackage{caption}
\usepackage{subcaption}
\usepackage[hidelinks]{hyperref}
\usepackage{enumerate}
\usepackage{amsmath}
\usepackage{amssymb}
\usepackage{mathrsfs}
\usepackage{mathtools}
\usepackage{appendix}
\usepackage{amsrefs}

\usepackage{url}
\usepackage{accents}
\usepackage{setspace}
\usepackage[margin=2.5cm]{geometry}
\usepackage{fancyhdr}

\allowdisplaybreaks
\BibSpec{article}{%
	+{}{\PrintAuthors} {author}
	+{,}{ } {title}
	+{, }{\textit } {journal}
	+{}{ \parenthesize} {date}
	+{,  }{no. } {volume}
	+{,}{ } {pages}
	+{,}{ } {note}
}

\ExplSyntaxOn

\ExplSyntaxOff
\BibSpec{book}{%
	+{}{\PrintAuthors}  {author}
	+{. }{}{title}
	+{,}{ }{series}
	+{,}{ vol.~}{volume}
	+{, }{\textit}{publisher}
	+{, }{}{date}
	+{. }{ } {note}
}
\BibSpec{collection.article}{%
	+{}{\PrintAuthors}{author}
	+{, }{}{title}
	+{, }{\textit}{booktitle}
	+{, }{ \DashPages}{pages}
	+{,}{ }{series}
	+{, }{}{volume}
	+{, }{\textit}{publisher}
	+{,}{ }{date}
}

\newcommand{\hcirc}{\accentset{\circ}{h}}
\newcommand{\hbarcirc}{\accentset{\circ}{\bar{h}}}
\newcommand{\htildecirc}{\accentset{\circ}{\tilde{h}}}

\newtheorem{prop}{Proposition}
\newtheorem{thm}[prop]{Theorem}
\newtheorem{lem}[prop]{Lemma}
\newtheorem{coro}[prop]{Corollary}
\newtheorem{rema}[prop]{Remark}

\title[Huisken-Yau for area-constrained Willmore spheres]{Huisken-Yau-type uniqueness for area-constrained Willmore spheres}

\author[Eichmair]{Michael Eichmair}
\address{
	\textnormal{Michael Eichmair \newline  \indent
		University of Vienna \newline \indent
		Faculty of Mathematics  \newline \indent
		Oskar-Morgenstern-Platz 1 \newline \indent
		1090 Vienna, 	Austria  \newline\indent 
		\href{https://orcid.org/0000-0001-7993-9536}{https://orcid.org/0000-0001-7993-9536} \newline\indent	
		\href{mailto:michael.eichmair@univie.ac.at}{michael.eichmair@univie.ac.at}}
}

\author[Koerber]{Thomas Koerber}
\address{
	\textnormal{Thomas Koerber \newline  \indent
		University of Vienna \newline \indent
		Faculty of Mathematics  \newline \indent
		Oskar-Morgenstern-Platz 1 \newline \indent
		1090 Vienna, 	Austria  \newline\indent 
		\href{https://orcid.org/0000-0003-1676-0824}{0000-0003-1676-0824}} \newline\indent	
		\href{mailto:thomas.koerber@univie.ac.at}{thomas.koerber@univie.ac.at}}

\author[Metzger]{Jan Metzger}
\address{
	\textnormal{Jan Metzger \newline  \indent
		 University of Potsdam\newline \indent
		Institute of Mathematics  \newline \indent
		Karl-Liebknecht-Straße 24-25 \newline \indent
			14476 Potsdam,  		Germany  \newline\indent 
		\href{https://orcid.org/0000-0001-5632-7916}{0000-0001-5632-7916}}\newline\indent	
	\href{mailto:jan.metzger@uni-potsdam.de}{jan.metzger@uni-potsdam.de}}

\author[Schulze]{Felix Schulze}
\address{
	\textnormal{Felix Schulze \newline  \indent
		University of Warwick\newline \indent
		Mathematics Institute  \newline \indent
		Coventry CV4 7AL,
		United Kingdom \newline\indent 
		\href{https://orcid.org/0000-0002-7011-2126}{0000-0002-7011-2126}}\newline\indent	
	\href{mailto:felix.schulze@warwick.ac.uk}{felix.schulze@warwick.ac.uk}}

\begin{document}

\date{\today}
\onehalfspacing
\begin{abstract} Let $(M,g)$ be a Riemannian three-manifold that is asymptotic to Schwarzschild. 
		We study the existence of large area-constrained Willmore spheres $\Sigma \subset M$ with non-negative Hawking mass and inner radius $\rho$ dominated by the area radius $\lambda$. If the scalar curvature of $(M,g)$ is non-negative, we show that no such surfaces with $\log \lambda \ll \rho$ exist. This answers a question of G.~Huisken.
	
\end{abstract}
\maketitle
\section{Introduction}
Let $(M,g)$ be a connected, complete Riemannian three-manifold. Let $\Sigma\subset M$ be a closed, two-sided surface with  area element $\mathrm{d}\mu,$ outward normal $\nu$,  and mean curvature $H$ with respect to $\nu$. The Hawking mass of $\Sigma$ is 
$$
m_H(\Sigma)=\sqrt{\frac{|\Sigma|}{16\,\pi}}\left(1-\frac{1}{16\,\pi}\int_{\Sigma} H^2\,\mathrm{d}\mu\right).
$$
In the case where $(M,g)$ arises as maximal  initial data for the Einstein field equations, the Hawking mass  has been proposed as a quasi-local measure for the strength of the gravitational field; see \cite{hawking1968gravitational}. 
 \\ \indent
 Recall that time-symmetric initial data for a Schwarzschild black hole with mass $m>0$ are given by 
  \begin{align}
 \label{schwarzschild} \bigg(\bigg\{x\in\mathbb{R}^3:|x|>\frac{m}{2}\bigg\}, \left(1+\frac{m}{2\,|x|}\right)^4\bar g\bigg)
 \end{align}
where $$\bar g=\sum_{i=1}^3dx^i\otimes d x^i$$ is  the Euclidean metric on $\mathbb{R}^3$ and $|x|$ denotes the Euclidean length of $x\in\mathbb{R}^3$.   A special class of general initial data consists of those asymptotic to Schwarzschild.  
 Given a non-negative integer $k$, we say  that $(M,g)$ is $C^k$-asymptotic to Schwarzschild with mass $m>0$ if there is a non-empty compact set $K\subset M$ such that the end $M\setminus K$ is diffeomorphic to $\{x\in\mathbb{R}^3\,:\,|x|>1\}$ and, in this special chart, there holds, as $x\to\infty$,
\[
g=\bigg(1+\frac{m}{2\,|x|}\bigg)^4\,\bar g+\sigma\qquad\text{where} \qquad \partial_J \sigma=O(|x|^{-2-|J|})
\]
for every multi-index $J$ with $|J|\leq k$.  Given $r>1$, we define $B_r\subset M$ to be the compact domain whose boundary corresponds to $S_r(0)$ in this chart. Moreover, given a closed, two-sided surface $\Sigma\subset M$, we define the area-radius $\lambda(\Sigma)>0$ and inner radius $\rho(\Sigma)$ of $\Sigma$ by  \begin{align} \label{inner and outer radius} 4\,\pi\, \lambda(\Sigma)^2=|\Sigma|\qquad \text{and}
\qquad
\rho(\Sigma)=\sup\{r>1\,:\, B_{r}\cap\Sigma=\emptyset\}.\end{align}
  \indent The Hawking mass $m_H(\Sigma)$ provides useful information on the strength of the gravitational field  provided the surface $\Sigma$ is either a  stable constant mean curvature sphere with large enclosed volume or, alternatively, an area-constrained Willmore sphere with large area; see also \cite[p.~2348]{bartniklocal}.  Stable constant mean curvature surfaces are stable critical points of the area functional under a volume constraint and therefore candidates to have least perimeter among surfaces of the same enclosed volume. S.-T.~Yau and D.~Christodoulou \cite{christodoulou71some} have observed that the Hawking mass of  stable constant mean curvature spheres is non-negative if $(M,g)$ has non-negative scalar curvature. Note that, in this context, the scalar curvature provides a lower bound for the energy density of the initial data set.    Stable constant mean curvature spheres have since been studied extensively in the context of mathematical relativity; see for example the recent overview given in \cite{cmc}.\\ \indent Recall that 
    $\Sigma\subset M$ is  an area-constrained Willmore surface if there is a number $\kappa\in\mathbb{R}$ such that
\begin{align}
\Delta H+(|\hcirc|^2+\operatorname{Ric}(\nu,\nu)+\kappa)\,H=0. \label{constrained Willmore equation}
\end{align}
Here, $\Delta$ is the non-positive Laplace-Beltrami operator on $\Sigma$ with respect to the induced metric, $\hcirc$  the traceless part of the second fundamental form $h$, and $\operatorname{Ric}$ the Ricci curvature of $(M,g)$. Note that \eqref{constrained Willmore equation} is the Euler-Lagrange equation of the Willmore energy
\begin{align} \label{Willmore energy intro} 
\int_{\Sigma} H^2\,\mathrm{d}\mu 
\end{align} 
with respect to an area constraint and $\kappa$ the corresponding Lagrange parameter. Area-constrained Willmore surfaces are therefore candidates to have largest Hawking mass among all closed surfaces of the same area.  As observed in \cite[p.~487]{acws2}, large area-constrained Willmore spheres capture  information on the asymptotic distribution of scalar curvature  that large   stable constant mean curvature spheres are impervious to.  
\subsection*{Existence and uniqueness of large area-constrained Willmore surfaces}
In the  recent papers \cite{acws,acws2}, the first-named author and the second-named author have studied the existence, uniqueness, and physical properties of large area-constrained Willmore spheres.  We recall the following result; see Figure \ref{acws thm figure}.
\begin{thm}[{\cite[Theorem 6]{acws2}}] \label{acws theorem}
	Let $(M,g)$ be $C^4$-asymptotic to Schwarzschild with mass $m>0$ and  suppose that
	\begin{align}
\sum_{i=1}^3	x^i\,\partial_i(|x|^2\,R)\leq 0 \label{growth condition}
	\end{align}
	outside a compact set. There exists $\kappa_0>0$ and a family \begin{align} \label{foliation} \{\Sigma(\kappa)\,:\,\kappa\in(0,\kappa_0)\} \end{align} 
	 of   spheres $\Sigma(\kappa)\subset M$ where $\Sigma(\kappa)$ satisfies \eqref{constrained Willmore equation} with parameter $\kappa$. The family \eqref{foliation} sweeps out the complement of a compact set in $M$ and there holds $m_H(\Sigma(\kappa))\geq 0$ for each $\kappa\in(0,\kappa_0)$.  \\ \indent
	Moreover, given $\delta>0$, there exists $\lambda>1$ and a compact set $K\subset M$  with the following property. If $\Sigma\subset M\setminus K$ is an area-constrained Willmore sphere with $m_H(\Sigma)\geq 0$ and $|\Sigma|>4\,\pi\,\lambda^2$, then either $\Sigma=\Sigma(\kappa)$ for some $\kappa\in(0,\kappa_0)$ or $\rho(\Sigma)<\delta\,\lambda(\Sigma)$. 
\end{thm}
\begin{rema}\text{ }
	\begin{enumerate}[i]
		\item[i)] Note that \eqref{growth condition} implies that $R\geq 0$ at infinity.
		\item[ii)]  T.~Lamm, the third-named author, and the fourth-named author have previously proved the existence of an asymptotic foliation by  large-area constrained Willmore spheres if $(M,g)$ is a so-called small perturbation of Schwarzschild; see \cite[Theorem 1 and Theorem 2]{lamm2011foliations}.
		\item[iii)] In $\mathbb{R}^3$, round spheres are Willmore surfaces and the only closed surfaces with non-negative Hawking mass; see \cite[(3)]{Willmore}.
		\item[iv)] S.~Brendle \cite{brendle2013constant} has shown that the spheres of symmetry are the only closed, embedded constant mean curvature surfaces in spatial Schwarzschild \eqref{schwarzschild}. It is not known if these are also the only area-constrained Willmore spheres; see also \cite[Remark 1.5 and Theorem 1.6]{Mondino}.
	\end{enumerate}
	
\end{rema}
	\begin{figure}\centering
	\includegraphics[width=0.3\linewidth]{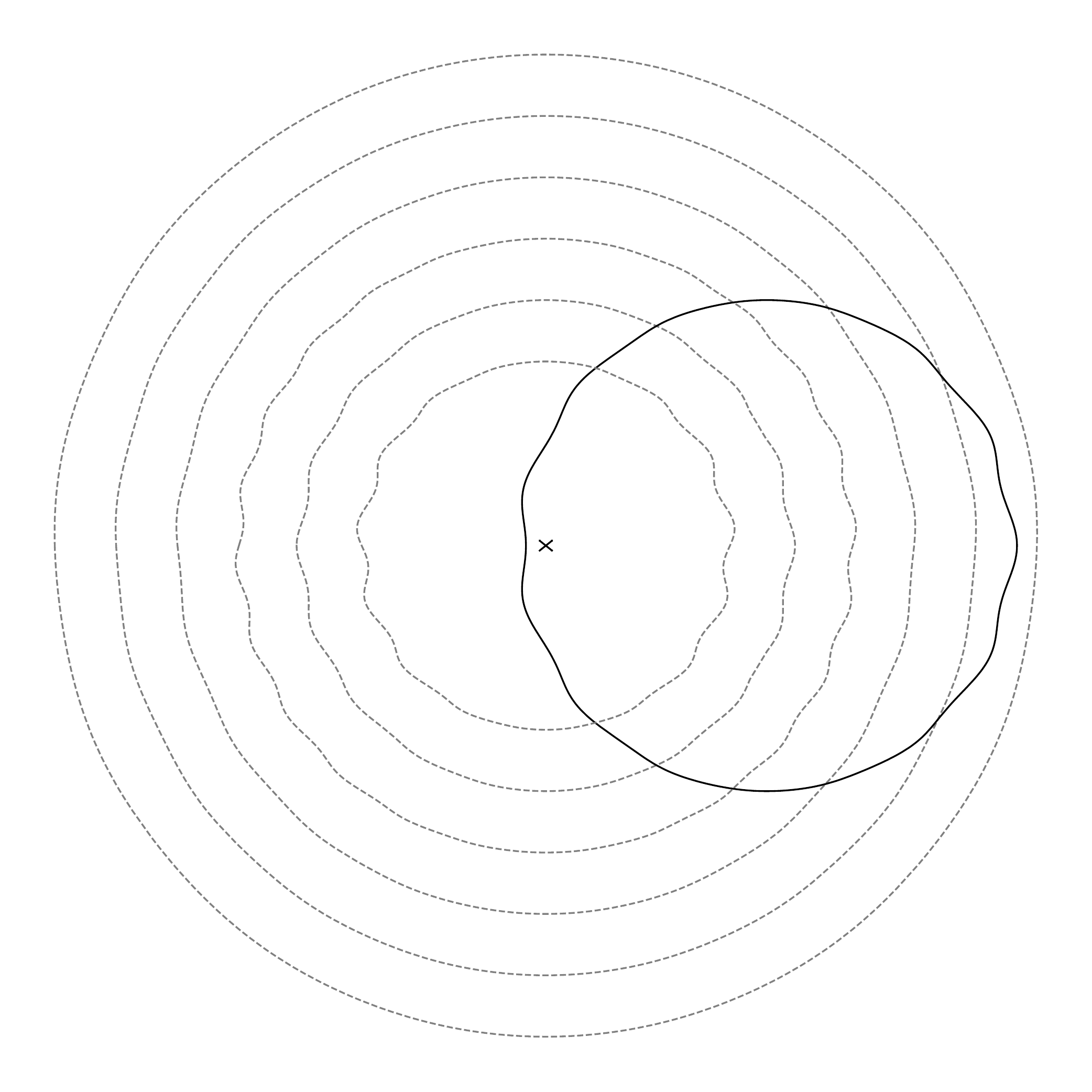}
	\caption{An illustration of the asymptotic family \eqref{foliation} by area-constrained Willmore spheres. The cross marks the origin in the asymptotically flat chart. The solid black line indicates a potential area-constrained Willmore sphere $\Sigma\subset M$ with $m_H(\Sigma)\geq0$ and $\rho(\Sigma)<\delta\,\lambda(\Sigma)$.  }
	\label{acws thm figure}
\end{figure}
As discussed for example in the introduction of \cite{acws}, the assumptions that the surfaces in consideration have non-negative Hawking mass, be large, and be disjoint from a certain bounded set appear to be essential for a characterization result such as Theorem \ref{acws theorem} to hold. By contrast, as we explain below, we conjecture that the alternative  $\rho(\Sigma)<\delta\,\lambda(\Sigma)$ in the conclusion of Theorem \ref{acws theorem} does not  actually arise. In fact, in this paper, we improve the uniqueness result in Theorem \ref{acws theorem} by ruling out the existence of certain large area-constrained Willmore spheres whose respective inner radius is small compared to their area radius.
\begin{thm}
	\label{main result} Let $(M,g)$ be $C^4$-asymptotic to Schwarzschild and suppose that, as $x\to\infty$ \begin{align}
	R\geq -o(|x|^{-4}). \label{scalar curvature lower bound}
	\end{align}
	There are $\delta>0$ and $\lambda>1$ with the following property. \\ \indent   There is no area-constrained Willmore sphere $\Sigma\subset M$ with 
	\begin{itemize}
\item[$\circ$] $m_H(\Sigma)\geq 0$,
\item[$\circ$] $|\Sigma|>4\,\pi\,\lambda^2$,
\item[$\circ$] $\rho(\Sigma)<\delta\,\lambda(\Sigma)$,
\item[$\circ$]  $\log\lambda(\Sigma)<\delta\,\rho(\Sigma)$.
	\end{itemize}

\end{thm} 
\begin{rema} \text{ } \label{sharp remark}
		\begin{enumerate}
		\item[i)] 	The conclusion of Theorem \ref{main result} can fail if the assumption \eqref{scalar curvature lower bound} is dropped; see \cite[Theorem 11]{acws2}.
		\item[ii)] Note that, for every $s>1$, $\log\lambda(\Sigma)<\lambda(\Sigma)^{1/s}$ provided that $\lambda(\Sigma)$ is sufficiently large.   Analytically,	the result established by G.~Huisken and S.-T.~Yau in \cite[Theorem 5.1]{HuiskenYau} on the uniqueness of large stable constant mean curvature spheres corresponds to the case where $1<s<2$. 
		\item[iii)]  The assumption  $\log\lambda(\Sigma)<\delta\,\rho(\Sigma)$ is essential to obtain \eqref{proj 0} from \eqref{optimal} and seems to be optimal for the method employed in this paper. Specifically, note that \eqref{optimal} is the best possible estimate based on  Lemma \ref{hcirc int etsimate}.
		\item[iv)]The assumptions of Theorem \ref{main result} imply that $\Sigma\cap B_2=\emptyset$. Note that, unlike in \cite{HuiskenYau} and \cite{QingTian}, we do not assume that $\Sigma$ encloses $B_2$.
	\end{enumerate}

\end{rema}

Combining Theorem \ref{main result} with Theorem \ref{acws theorem}, we obtain the following corollary.
\begin{coro} \label{uniqueness}
	Let $(M,g)$ be $C^4$-asymptotic to Schwarzschild. Suppose that
	$$
	\sum_{i=1}^3x^i\,\partial_i(|x|^2\,R)\leq 0
	$$ outside a compact set.
	Let $\{\Sigma_i\}_{i=1}^\infty$ be a sequence of area-constrained Willmore spheres $\Sigma_i\subset M$  with
	\begin{itemize}
		\item[$\circ$] $m_H(\Sigma)\geq 0$,
		\item[$\circ$] $\lim_{i\to\infty}\rho(\Sigma_i)=\infty$,
		\item[$\circ$] $\lim_{i\to\infty}\lambda(\Sigma_i)=\infty$,
		\item[$\circ$] $\Sigma_i$ is not part of the foliation \eqref{foliation}.
	\end{itemize}
There holds 
	$
	\rho(\Sigma_i)=O(\log\lambda(\Sigma_i)).
	$
\end{coro}
For Riemannian three-manifolds  asymptotic to Schwarzschild and satisfying \eqref{growth condition}, Corollary \ref{uniqueness} provides evidence that all area-constrained Willmore spheres  with large inner radius, large area radius, and non-negative Hawking mass belong to the family \eqref{foliation}. In addition, it stands to reason that, for each $\kappa\in(0,\kappa_0)$, the Hawking mass of $\Sigma(\kappa)$ is maximal among all  spheres $\Sigma\subset M$ with $|\Sigma|=|\Sigma(\kappa)|$ provided that $\rho(\Sigma)$ is sufficiently large. We note that both of these conjectures are open, even in the case where $(M,g)$ is the spatial Schwarzschild manifold \eqref{schwarzschild}. By contrast, it is known that, for such Riemannian three-manifolds, stable constant mean curvature spheres with large area are the unique solutions of the isoperimetric problem for the volume they enclose; see the work of O.~Chodosh and the first-named author \cite{chodosh2017global} and of O.~Chodosh, Y.~Shi, H.~Yu, and the first-named author \cite{CESH}.  
\subsection*{Outline of related results} We say that a sequence $\{\Sigma_i\}_{i=1}^\infty$ of spheres $\Sigma_i\subset M$ with
\begin{align} \label{slow divergence intro}
\lim_{i\to\infty}\rho(\Sigma_i)=\infty \qquad \text{and}\qquad \rho(\Sigma_i)=o(\lambda(\Sigma_i))
\end{align}
is slowly divergent. As with large stable constant mean curvature spheres, a substantial obstacle towards establishing the uniqueness of large area-constrained Willmore spheres with non-negative Hawking mass in  Riemannian three-manifolds asymptotic to Schwarzschild is to rule out  the possibility of a  slowly divergent sequence $\{\Sigma_i\}_{i=1}^\infty$ of area-constrained Willmore spheres $\Sigma_i\subset M$ with non-negative Hawking mass.  The main difficulty in understanding the geometry of the spheres $\Sigma_i$ owes to the fact that unrefined curvature estimates generally do not yield global analytic control. In fact, as $i\to\infty$, there holds
\begin{align} \label{slow divergence curv est intro} 
h(\Sigma_i)=O(\lambda(\Sigma_i)^{-1})+O((\lambda(\Sigma_i)^{-1/2}+\rho(\Sigma_i)^{-1})\,|x|^{-1});
\end{align}
see Proposition \ref{curv est prop}. If for example $\rho(\Sigma_i)=o(\lambda(\Sigma_i)^{1/2})$, estimate \eqref{slow divergence curv est intro} fails to bound the sequence $\{\lambda(\Sigma_i)^{-1}\,\Sigma_i\}_{i=1}^\infty$ in $C^2$. If for example $\rho(\Sigma_i)=o(\log\lambda(\Sigma_i))$, \eqref{slow divergence curv est intro} even fails to bound the sequence $\{\lambda(\Sigma_i)^{-1}\,\Sigma_i\}_{i=1}^\infty$ in $C^1$; see Figure \ref{Figure slow divergence}. 
\begin{figure}
	\centering
	\begin{subfigure}{0.33\textwidth}
		
		\includegraphics[width=1\linewidth]{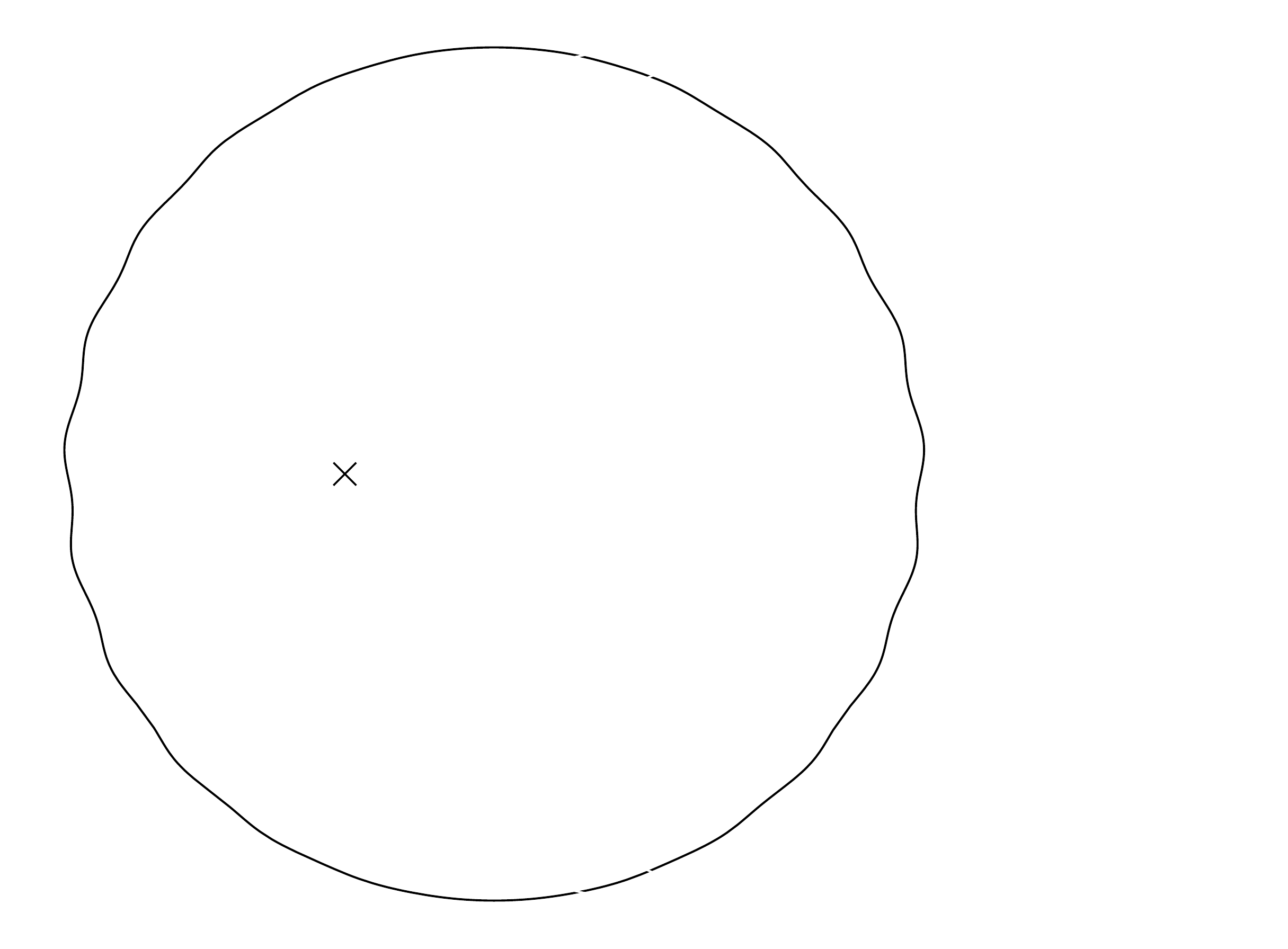}
		
	\end{subfigure}%
	\begin{subfigure}{0.33\textwidth}
		
		\includegraphics[width=1\linewidth]{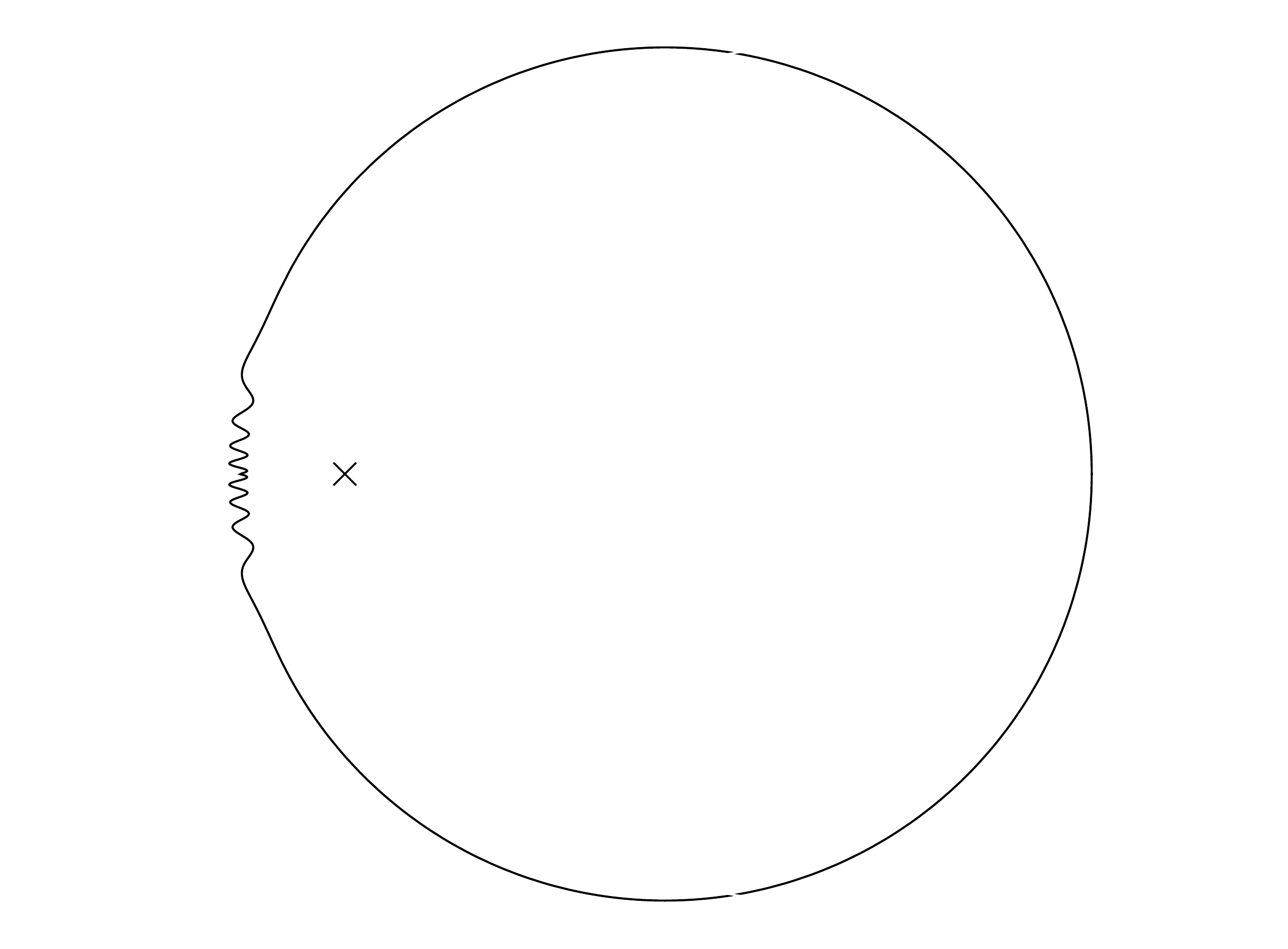}
	\end{subfigure}
	\begin{subfigure}{0.33\textwidth}
		
		\includegraphics[width=1\linewidth]{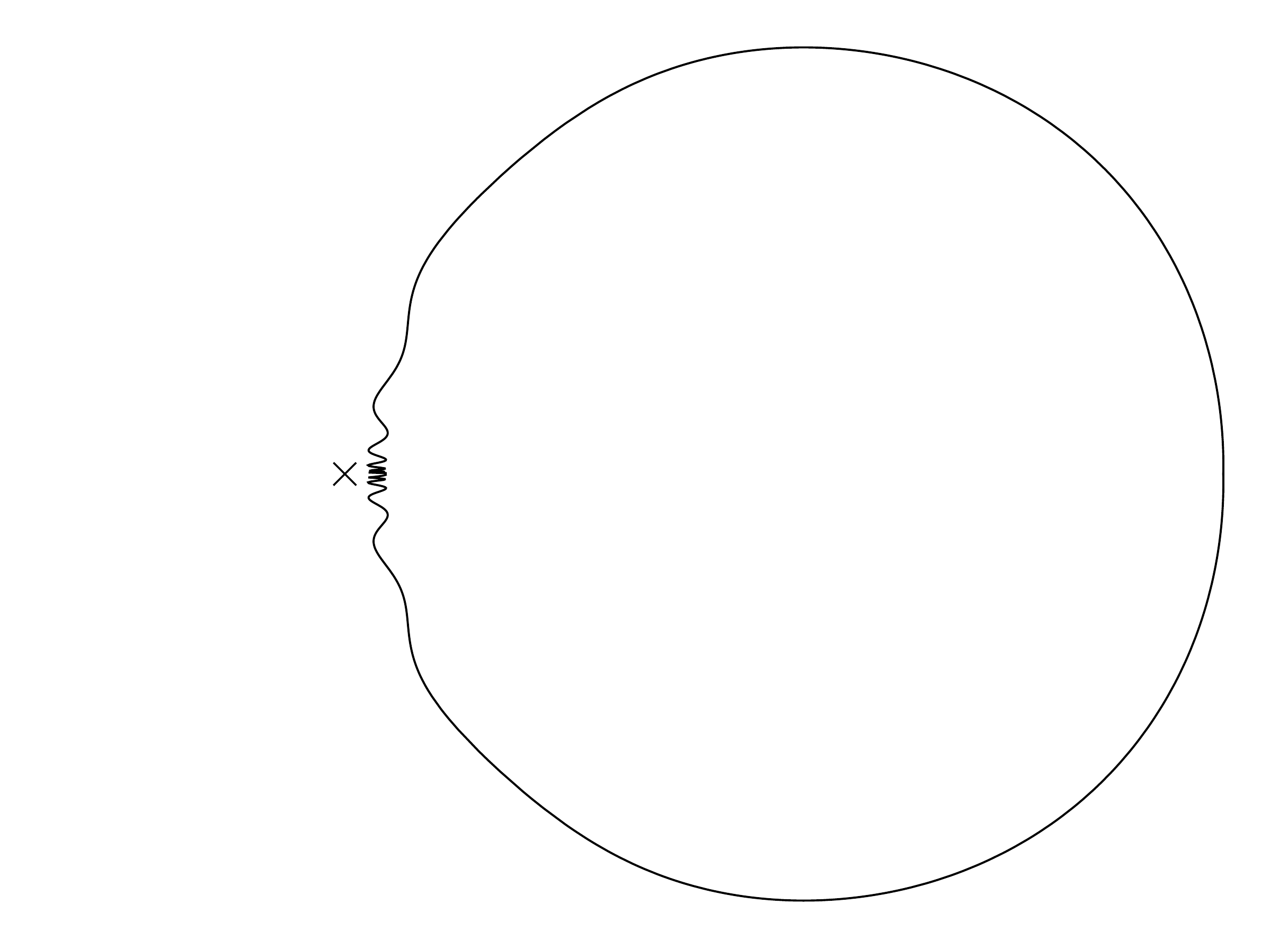}
	\end{subfigure}
	\caption{An illustration of a slowly divergent sequence $\{\Sigma_i\}_{i=1}^\infty$ of area-constrained Willmore spheres $\Sigma_i\subset M$ on the scale of the area radius $\lambda(\Sigma_i)$. The cross marks the origin in the asymptotically flat chart.  Away from the origin, the surfaces $\Sigma_i$ converge uniformly to a round sphere.}
	\label{Figure slow divergence}
\end{figure}
 \\ \indent 
 G.~Huisken and S.-T.~Yau \cite{HuiskenYau} have shown that  there are no  slowly divergent sequences $\{\Sigma_i\}_{i=1}^\infty$ of   stable constant mean curvature spheres $\Sigma_i\subset M$ that enclose $B_2$ with $\lambda(\Sigma_i)=O(\rho(\Sigma_i)^s)$ where $1<s<2$. In this case, an estimate similar to \eqref{slow divergence curv est intro} provides uniform estimates in $C^2$. These estimates are sufficient to conclude their argument  based on analyzing a certain flux integral related to the variation of the area functional with respect to a translation.  Following the same strategy, J.~Qing and G.~Tian \cite{QingTian} have shown that the assumption $\lambda(\Sigma_i)=O(\rho(\Sigma_i)^s),$ $1<s<2$, can be dropped. To overcome the potential loss of $C^1$-control, they carry out an asymptotic analysis  based on the observation that the Gauss maps $\{\nu(\Sigma_i)\}_{i=1}^\infty$ form a sequence of almost harmonic maps. Finally, O.~Chodosh and the first-named author \cite{chodosh2017global} have shown that the assumption that $\Sigma_i$ encloses $B_2$ can be dropped if the scalar curvature of $(M,g)$ is non-negative. Their method is based on an analysis of the Hawking mass of $\Sigma_i$. To obtain the required analytic control, they combine the Christodoulou-Yau estimate \cite[p.~13]{christodoulou71some}
 \begin{align} \label{CY}
 \frac23\,\int_{\Sigma}(|\hcirc|^2+R)\,\mathrm{d}\mu\leq 16\,\pi-\int_{\Sigma}H^2\,\mathrm{d}\mu,
 \end{align}
valid for every stable constant mean curvature sphere $\Sigma\subset M$, with global methods developed by G.~Huisken and T.~Ilmanen \cite{HI}. We also refer to the papers of L.-H.~Huang \cite{Huang}, of S.~Ma \cite{Ma}, and of the first-named author and second-named author \cite{cmc} on slowly divergent sequences of large stable constant mean curvature spheres in general asymptotically flat Riemannian three-manifolds.
\\ \indent 
When studying slowly divergent sequences \eqref{slow divergence intro} of area-constrained Willmore spheres with non-negative Hawking mass, additional difficulties arise. On the one hand, the absence of an estimate comparable to  \eqref{CY} renders the curvature estimates for large area-constrained Willmore spheres less powerful than those for large stable constant mean curvature spheres. Moreover, the fourth-order nature of the area-constrained Willmore equation \eqref{constrained Willmore equation} poses additional analytical challenges. On the other hand, the variation of the Willmore energy \eqref{Willmore energy intro} with respect to a translation is of a  smaller scale than that of the area functional, at least when $\Sigma$ encloses $B_2$. Consequently, more precise analytic control is needed. In fact, we are not aware of any previous positive results on the non-existence of slowly divergent sequences of  area-constrained Willmore spheres with non-negative Hawking mass. By contrast, in \cite[Theorem 11]{acws2}, the first-named author and the second-named author have shown that such sequences may exist if the scalar curvature of $(M,g)$ is allowed to change sign.

\subsection*{Outline of the proof of Theorem \ref{main result}}
By scaling, we may assume that $m=2$, that is,
$$
g=(1+|x|^{-1})^4\,\bar g +\sigma.
$$  
We use a bar to indicate that a geometric quantity has been computed with respect to the Euclidean background metric $\bar g$. Likewise, we use a tilde to indicate that the Schwarzschild metric
$$
\tilde g=(1+|x|^{-1})^4\,\bar g
$$
with mass $m=2$ has been used in the computation.
\\ \indent 
  Assume that $\{\Sigma_i\}_{i=1}^\infty$ is a sequence of area-constrained Willmore spheres $\Sigma_i\subset M$ with $m_H(\Sigma_i)\geq 0$ and
\begin{align} \label{s intro} 
\lim_{i\to\infty}\rho(\Sigma_i)=\infty,\qquad\rho(\Sigma_i)=o(\lambda(\Sigma_i)),\qquad  
\log\lambda(\Sigma_i)=o(\rho(\Sigma_i)).
\end{align}
 Let  $a\in\mathbb{R}^3$. 
To prove Theorem \ref{main result}, we expand the variation of the Willmore energy of $\Sigma_i$ with respect to a translation in direction $a$ given by
\begin{align} \label{translation sensitivity}
0=-\int_{\Sigma_i} g(a,\nu)\,\left[\Delta H+(|\hcirc|^2+\operatorname{Ric}(\nu,\nu)+\kappa)\,H\right]\mathrm{d}\mu.
\end{align} 
Contrary to the variation of the area functional with respect to a translation, we expect the right-hand side of \eqref{translation sensitivity} to be small 	independently of whether $\Sigma_i$ encloses the origin or not; see \eqref{div theorem}. Consequently, precise analytic control is needed to expand the terms on the right-hand side of \eqref{translation sensitivity} with sufficient control on the error. 
 \\ \indent To this end, we first prove preliminary pointwise curvature  estimates for large area-constrained Willmore spheres with non-negative Hawking mass; see Proposition \ref{curv est prop}. These estimates are based on an adaptation of the localized integral curvature estimates proved by E.~Kuwert and R.~Schätzle \cite{kuwertSchaetzle2} for Willmore surfaces in $\mathbb{R}^3$ to the setting of large area-constrained Willmore spheres in Riemannian three-manifolds asymptotic to Schwarzschild, see Appendix \ref{int curv est appendix}, and an $L^2$-curvature estimate that builds on an argument given in \cite[\S4]{chodosh2017global}. Combining Proposition \ref{curv est prop} with \eqref{s intro}, it follows that $\Sigma_i$ is the radial graph of a function $u_i$ over a large coordinate sphere $S_i=S_{\lambda_i}(\lambda_i\,\xi_i)$ where $\lambda_i>1$ and $\xi_i\in\mathbb{R}^3$; see  Lemma \ref{u initial estimate}. The resulting estimates in Lemma \ref{u initial estimate 3} below are still too coarse to expand \eqref{translation sensitivity}. To overcome this, we first prove explicit estimates for the Laplace operator of a round sphere based on Green's function methods; see Lemma \ref{PDE lemma 2}. Second, we observe that the quotient of $H(\Sigma_i)$ and the potential function of the spatial Schwarzschild manifold satisfies an equation slightly more useful than the area-constrained Willmore equation \eqref{condition}; see Lemma \ref{potential function}. We then use the explicit estimates for the Laplace operator to investigate this equation to obtain the sharp estimate
 \begin{align} \label{intro H estimate}
 H(\Sigma_i)=(2+o(1))\,\lambda(\Sigma_i)^{-1}-4\,\lambda(\Sigma_i)^{-1}\,|x|^{-1}+o(\lambda(\Sigma_i)^{-1}\,\rho(\Sigma_i)^{-1});
 \end{align} 
 see Lemma \ref{improved H estimate}. We note that this procedure requires assumption \eqref{s intro} in an essential way; see Remark \ref{sharp remark}.
 \\ \indent
  With the estimate \eqref{intro H estimate} at hand, we obtain that   
  \begin{align*} 
0= \,& -\int_{\Sigma_i} g(\xi_i,\nu)\,\left[\Delta H+(|\hcirc|^2+\operatorname{Ric}(\nu,\nu)+\kappa)\,H\right]\mathrm{d}\mu\\=\,&8\,\pi\,\lambda(\Sigma_i)^{-1}\,\rho(\Sigma_i)^{-2}-	\lambda(\Sigma_i)^{-1}\,\int_{S_i} \bar g(\xi_i,\bar\nu)\,R\,\mathrm{d}\bar \mu-o(\lambda(\Sigma_i)^{-1}\,\rho(\Sigma_i)^{-2});
  \end{align*} 
  see \eqref{contradiction}.  Using that $R\geq-o(|x|^{-4})$, it follows that 
  $$
  0 \geq 8\,\pi\,\lambda(\Sigma_i)^{-1}\,\rho(\Sigma_i)^{-2}-o(\lambda(\Sigma_i)^{-1}\,\rho(\Sigma_i)^{-2}).
  $$
This is a contradiction for  sufficiently large $i$.
\subsection*{Acknowledgments} We thank Gerhard Huisken for proposing the question answered in this paper and for his encouragement and support. We thank the referees for their feedback which helped improve the exposition of this paper. Michael Eichmair acknowledges the support of the START Programme Y963 of the Austrian Science Fund. Thomas Koerber acknowledges the support of the Lise-Meitner Programme M3184 of the Austrian Science Fund. Jan Metzger was supported by the DFG Project ME 3816/3-1 which is part of the
SPP2026. Part of the work was done during his participation in the workshop General
Relativity, Geometry and Analysis: beyond the first 100 years after
Einstein at Institut Mittag-Leffler.\\ \indent For the purpose of open access, the authors have applied a Creative Commons Attribution (CC-BY) license to any Author Accepted Manuscript version arising from this submission.
\section{Curvature estimates for large area-constrained Willmore spheres} \label{curv est appendix}
We assume that $g$ is a Riemannian metric on $\mathbb{R}^3$ such that, as $x\to\infty,$
$$
g=\left(1+|x|^{-1}\right)^4\,\bar g+\sigma \qquad\text{where}\qquad \partial_J\sigma+O(|x|^{-2-|J|})
$$  for every multi-index $J$ with $|J|\leq 4$. \\ \indent 
Let $\{\Sigma_i\}_{i=1}^\infty$ be a sequence of area-constrained Willmore spheres $\Sigma_i\subset \mathbb{R}^3$ with  
\begin{align} \label{appendix b assumptions}
	\int_{\Sigma_i}H^2\,\mathrm{d}\mu\leq 16\,\pi,\qquad\lim_{i\to\infty}\rho(\Sigma_i)=\infty,\qquad \rho(\Sigma_i)=O(\lambda(\Sigma_i)).
\end{align} 
\indent 
The goal of this section is to prove  curvature estimates for $\Sigma_i$. To this end, we combine the integral curvature estimates from Appendix \ref{int curv est appendix} with the integral estimate on the second fundamental form proven in Lemma \ref{hcirc int etsimate}. \\ \indent
We abbreviate $\rho_i=\rho(\Sigma_i)$ and $\lambda_i=\lambda(\Sigma_i)$. 
\begin{lem}
As $i\to\infty$, there holds: \label{large blowdown} \label{L2 curv est rema}
	\begin{align*}
		&\circ\qquad 	\int_{\Sigma_i}\bar H^2\,\mathrm{d}\bar\mu=16\,\pi+O(\rho_i^{-1})\\
		&\circ\qquad		\int_{\Sigma_i}|h|^2\,\mathrm{d}\mu=O(1) \\
		&\circ\qquad		\int_{\Sigma_i}|\bar h|_{\bar g}^2\,\mathrm{d}\bar\mu=O(1)
	\end{align*}
\end{lem}
\begin{proof}
	Clearly,
	$$
	\int_{\Sigma_i}\bar H^2\,\mathrm{d}\bar\mu\geq 16\,\pi.
	$$
	Integrating the Gauss equation and using the Gauss-Bonnet theorem, we have
	$$\int_{\Sigma_i}H^2\,\mathrm{d}\mu=8\,\pi+\int_{\Sigma_i}|h|^2\,\mathrm{d}\mu +4\,\int_{\Sigma_i}\bigg(\operatorname{Rc}(\nu,\nu)-\frac12\,R\bigg) \,\mathrm{d}\mu.
	$$
	Using Lemma \ref{hy integral lemma} and \eqref{appendix b assumptions}, we obtain
	\begin{align} \label{intermediate} 
		\int_{\Sigma_i}|h|^2\,\mathrm{d}\mu\leq 8\,\pi+O(\rho_i^{-1})\,\int_{\Sigma_i}\bar H^2\,\mathrm{d}\bar\mu.
	\end{align} 
	By Lemma \ref{basic expansions},
	$$
	\bar H^2\,\mathrm{d}\bar\mu=H^2\,\mathrm{d}\mu+O(|x|^{-1}\,|h|^2)\,\mathrm{d}\mu+O(|x|^{-3})\,\mathrm{d}\bar\mu
	$$ 
	Using Lemma \ref{hy integral lemma} and \eqref{intermediate}, we obtain 
	$$
	\int_{\Sigma_i}\bar H^2\,\mathrm{d}\bar \mu\leq 16\,\pi+O(\rho_i^{-1})+ O(\rho_i^{-1})\,\int_{\Sigma_i}\bar H^2\,\mathrm{d}\bar\mu.
	$$
	\indent The assertion follows from these estimates.
\end{proof} 

\begin{lem}
	There holds \label{crude minkowski}
	$$	\int_{\Sigma_i}(\bar H-2\,\lambda_i^{-1})^2\,\mathrm{d}\bar\mu=O(1)\,\int_{\Sigma_i}|\hbarcirc|_{\bar g}^2\,\mathrm{d}\bar\mu+O(\rho_i^{-2}).
	$$
\end{lem}
\begin{proof}
	By {\cite[(38)]{dLM}}, 
	$$
	\int_{\Sigma_i}(\bar H-2\,\bar\lambda(\Sigma_i)^{-1})^2\,\mathrm{d}\bar\mu=O(1)\,\int_{\Sigma_i}|\hbarcirc|_{\bar g}^2\,\mathrm{d}\bar\mu.
	$$
	From Lemma \ref{basic expansions} we obtain 
	$
	\lambda_i=(1+O(\rho_i^{-1}))\,\bar\lambda(\Sigma_i)
	$
	from which the assertion follows. 
\end{proof}

\begin{lem}
	There holds \label{hcirc int etsimate}
	\begin{align*}
		\int_{\Sigma_i}|h-\lambda_i^{-1}\,g|_{\Sigma_i}|^2\,\mathrm{d}\mu=O((\lambda_i^{-1/2}+\rho_i^{-1})^2).
	\end{align*} 
	
\end{lem}
\begin{proof}
	Using Lemma \ref{basic expansions}, Lemma \ref{hy integral lemma}, and Lemma \ref{large blowdown}, we have
	$$
	\int_{\Sigma_i}H^2\,\mathrm{d}\mu=\int_{\Sigma_i}\tilde  H^2\,\mathrm{d}\tilde \mu+O(\rho_i^{-2})
	$$
	and
	$$
	\int_{\Sigma_i}\tilde H^2\,\mathrm{d}\tilde \mu=\int_{\Sigma_i}[\bar H^2-8\,(1+|x|^{-1})^{-1}\,|x|^{-3}\,\bar g(x,\bar\nu)\,\bar H
	]\,\mathrm{d}\bar \mu+O(\rho_i^{-2}).
	$$
	By the Gauss-Bonnet theorem,
	$$
	\int_{\Sigma_i}\bar H^2\,\mathrm{d}\bar\mu=16\,\pi+2\,\int_{\Sigma}|\hbarcirc|_{\bar g}^2\,\mathrm{d}\bar\mu.$$
	Moreover, by Lemma \ref{hy integral lemma} and Lemma \ref{crude minkowski}, we have
	\begin{align*} 
			\left|\int_{\Sigma_i}(1+|x|^{-1})^{-1}\,|x|^{-3}\,\bar g(x,\bar\nu)\,\bar H\,\mathrm{d}\bar \mu\right|&= 	\left|\int_{\Sigma_i}|x|^{-3}\,\bar g(x,\bar\nu)\,\bar H\,\mathrm{d}\bar \mu\right|+O(\rho_i^{-2})\\& \leq 2\,\lambda_i^{-1}\,\left|\int_{\Sigma_i}\,|x|^{-3}\,\bar g(x,\bar\nu)\,\mathrm{d}\bar \mu\right| +\frac18\,\int_{\Sigma_i}| \hbarcirc|_{\bar g}^2\,\mathrm{d}\bar\mu+O(\rho_i^{-2}).
	\end{align*} 
	Note that $\bar{\operatorname{div}}(|x|^{-3}\,x)=0$. Using the divergence theorem, we find that
	$$
	\int_{\Sigma_i}\,|x|^{-3}\,\bar g(x,\bar\nu)\,\mathrm{d}\bar \mu=4\,\pi
	$$
	if $\Sigma_i$ encloses $B_2$ and
	$$
	\int_{\Sigma_i}\,|x|^{-3}\,\bar g(x,\bar\nu)\,\mathrm{d}\bar \mu=0
	$$
	otherwise. In conjunction with \eqref{appendix b assumptions}, these estimates imply that
	\begin{align*}  
		\int_{\Sigma_i}|\hbarcirc|_{\bar g}^2\,\mathrm{d}\bar \mu=O(\lambda_i^{-1})+O(\rho_i^{-2}).
	\end{align*} 
	Using  this, Lemma \ref{basic expansions}, Lemma \ref{hy integral lemma}, and Lemma \ref{large blowdown}, we conclude that
	\begin{align*} 
		\int_{\Sigma_i}|\hcirc|^2\,\mathrm{d}\mu=\int_{\Sigma_i}|\hbarcirc|_{\bar g}^2\,\mathrm{d}\bar\mu+O(\rho_i^{-1})	\int_{\Sigma_i}|\hbarcirc|_{\bar g}^2\,\mathrm{d}\bar\mu+O(\rho_i^{-2})
		=O(\lambda_i^{-1})+O(\rho_i^{-2}).
	\end{align*} 
	Likewise, by Lemma \ref{basic expansions}, Lemma \ref{crude minkowski}, Lemma \ref{L2 curv est rema}, and Lemma \ref{hy integral lemma}, we have
	\begin{align*} 
		\int_{\Sigma_i}(H-2\,\lambda_i^{-1})^2\,\text{d}\mu\leq\,& \int_{\Sigma_i}(\bar H-2\,\lambda_i^{-1})^2\,\text{d}\bar\mu+O(\rho_i^{-2})
		\\\leq\,&O(1)\,\int_{\Sigma_i}|\hbarcirc|^2\,\mathrm{d}\bar\mu+O(\rho_i^{-2})
		\\\leq\, &O(\lambda_i^{-1})+O(\rho_i^{-2}).
	\end{align*} 
	\indent 
	Using that
	$$
	\int_{\Sigma_i}|h-\lambda_i^{-1}\,g|_{\Sigma_i}|^2\,\mathrm{d}\mu=\int_{\Sigma_i}(H-2\,\lambda_i^{-1})^2\,\mathrm{d}\mu+\int_{\Sigma_i}|\hcirc|^2\,\mathrm{d}\mu,
	$$  
	the assertion follows.
\end{proof} 
\begin{lem} \label{diameter in terms of area radius} There holds $
		\sup_{x\in \Sigma_i} |x|=O(\lambda_i).$
\end{lem}
\begin{proof} The assertion follows from \eqref{diam estimate}, \eqref{appendix b assumptions}, and Lemma \ref{large blowdown}.
	\end{proof}  
For the proof of Proposition \ref{curv est prop}, note that, by  Lemma \ref{hcirc int etsimate}, the sequence $\{\Sigma_i\}_{i=1}^\infty$ satisfies assumption \eqref{small curvature} of Appendix \ref{int curv est appendix}.
\begin{prop}
 As $i\to\infty$, there holds \label{curv est prop}
	\begin{align}
	\kappa(\Sigma_i)=O((\lambda_i^{-1/2}+\rho_i^{-1})\,\lambda_i^{-2}) \label{kappa est}
\end{align}
and
	\begin{align}
		h-\lambda_i^{-1}\,g|_{\Sigma_i}=O((\lambda_i^{-1/2}+\rho_i^{-1})\,|x|^{-1}).\label{H est}
	\end{align}
In particular, $|\hcirc|^2=O((\lambda_i^{-1/2}+\rho_i^{-1})^2\,|x|^{-2})$.	
\end{prop}
\begin{proof}
	We choose $\psi\in C^\infty(\mathbb{R})$ with
	\begin{itemize}
		\item[$\circ$] 	 $0\leq\psi\leq 1$,
		\item[$\circ$] $\psi(s)=1$ if $s\geq 1/4$,
		\item[$\circ$] $\psi(s)=0$ if $s\leq 1/8$.
	\end{itemize}
	We define
	$\gamma_i\in C^\infty(\mathbb{R}^3)$ by
	$$
	\gamma_i(x)=\psi(\lambda_i^{-1}\,|x|).
	$$
	Using Lemma \ref{hcirc int etsimate}, we have
	$$
	\int_{\Sigma_i}\gamma_i\,H\,\mathrm{d}\mu=2\,\lambda_i^{-1}\,\int_{\Sigma_i} \gamma_i\,\mathrm{d}\mu +o(\lambda_i).
	$$
	\indent If $\rho_i\geq1/4\,\lambda_i$, then $\gamma_i(x)=1$ for all $x\in \Sigma_i$. Consequently,
	$$
	\int_{\Sigma_i}\gamma_i\,\mathrm{d}\mu=4\,\pi\,\lambda_i^2.
	$$
	\indent If $\rho_i\leq1/4\,\lambda_i$, then we choose $x_i\in\Sigma_i$ with $|x_i|=\rho_i$ and apply \eqref{area estimate} with $r=\lambda_i/4+\rho_i$. Using Lemma \ref{large blowdown} and Lemma \ref{basic expansions}, we obtain
	$$
	\int_{\Sigma_i} \gamma_i\,\mathrm{d}\mu\geq 4\,\pi\,\lambda_i^2-(16\,\pi+o(1))\,\frac{3+2\,\sqrt{2}}{16}\,(\lambda_i/4+\rho_i)^2\geq \frac{9\,\pi}{4}\,\lambda_i^2
	$$
	for all $i$ sufficiently large.  \\ \indent 
	Either way, it follows that
	\begin{align} \label{mc integrals}
		4\,\pi\,\lambda_i\leq \int_{\Sigma_i}\gamma_i\,H\,\mathrm{d}\mu\leq \int_{\Sigma_i}|H|\,\mathrm{d}\mu\leq 16\,\pi \,\lambda_i
	\end{align}
	for all  sufficiently large $i$. Using  \eqref{constrained Willmore equation}, we have
	\begin{align*}
			-\kappa(\Sigma_i)\,\int_{\Sigma_i}\gamma_i\,H\,\mathrm{d}\mu=\int_{\Sigma_i}(\Delta\gamma_i)\,H\,\mathrm{d}\mu+\int_{\Sigma_i}\gamma_i\,H\,|\hcirc|^2\,\mathrm{d}\mu+\int_{\Sigma_i}\gamma_i\, \operatorname{Rc}(\nu,\nu)\,H\,\mathrm{d}\mu.
	\end{align*}
	Using \eqref{mc integrals}, we have
	\begin{align*}
		|\kappa(\Sigma_i)|\,\int_{\Sigma_i}\gamma_i\,H\,\mathrm{d}\mu\geq4\,\pi\,\lambda_i\,|\kappa(\Sigma_i)|.
	\end{align*}  
	Note that
	\begin{align*}
		\nabla^2\gamma_i=O(\lambda_i^{-1}\,|h|)+O(\lambda_i^{-2}).
	\end{align*} 
	In conjunction with Lemma \ref{hcirc int etsimate} and Lemma \ref{L2 curv est rema}, we obtain that
	\begin{align*}
		\int_{\Sigma_i}(\Delta\gamma_i)\,H\,\mathrm{d}\mu=\int_{\Sigma_i}(\Delta\gamma_i)\,(H-2\,\lambda_i^{-1})\,\mathrm{d}\mu=O((\lambda_i^{-1/2}+\rho_i^{-1})\,\lambda_i^{-1}).
	\end{align*}
	By Proposition \ref{int curv estimate} and Lemma \ref{hcirc int etsimate},  we have
	$$
	|\hcirc|^2=O((\lambda_i^{-1/2}+\rho_i^{-1})^2\,\lambda_i^{-2})+o(\kappa(\Sigma_i))
	$$
	on $\Sigma_i\cap\operatorname{spt}(\gamma_i)$. In conjunction with \eqref{mc integrals}, we obtain
	$$
	\int_{\Sigma_i}\gamma_i \,H\,|\hcirc|^2\,\mathrm{d}\mu=O((\lambda_i^{-1/2}+\rho_i^{-1})^2\,\lambda_i^{-1})+o(\kappa(\Sigma_i)\,\lambda_i).
	$$
	Likewise, \eqref{mc integrals} gives
	$$
	\int_{\Sigma_i}\gamma_i\,\operatorname{Rc}(\nu,\nu)\,H\,\mathrm{d}\mu=O(\lambda_i^{-2}).
	$$
	\eqref{kappa est} follows from these estimates. \\ \indent 
	Using \eqref{kappa est} and Proposition \ref{int curv estimate}, we see that 
	\begin{align*}
		\bigg(\int_{\Sigma_i}|h-\lambda_i^{-1}\,g|_{\Sigma_i}|^2\,\mathrm{d}\mu\bigg)^2=O((\lambda_i^{-1/2}+\rho_i^{-1})^4)
	\end{align*}  
	and
	\begin{align*}
		\kappa(\Sigma_i)^2\,	  	\int_{\Sigma_i}|h-\lambda_i^{-1}\,g|_{\Sigma_i}|^2\,\mathrm{d}\mu=O((\lambda_i^{-1/2}+\rho_i^{-1})^4\,\lambda_i^{-4}).
	\end{align*}  
		   Now, \eqref{H est}  follows from Proposition \ref{int curv estimate} using Lemma \ref{diameter in terms of area radius} and that $|x|\geq \rho_i$ for every $x\in \Sigma_i$. 
\end{proof}
 
\begin{prop} \label{large smooth blowodwn}
	A subsequence of $\{\lambda_i^{-1}\,\Sigma_i\}_{i=1}^\infty$ converges to a round sphere in $C^2$ locally  in $\mathbb{R}^3\setminus\{0\}$. 
\end{prop}
\begin{proof}
	Let $\hat\Sigma_i=\lambda_i^{-1}\,\Sigma_i$. By Proposition \ref{curv est prop} and Lemma \ref{basic expansions}, we have	
	\begin{align} \label{smooth large blowdown curv est} 
		\bar h(\hat\Sigma_i)-\bar g|_{\hat\Sigma_i}
		=o(1)
	\end{align} 
	locally uniformly in $\mathbb{R}^3\setminus\{0\}$. Let $x_i\in \hat\Sigma_i$ with
	$$
	|x_i|=\sup\{|y|:y\in \hat\Sigma_i\}.
	$$
	By Lemma \ref{diameter in terms of area radius}, there is $x\in\mathbb{R}^3$ such that, passing to a subsequence,
	$$
	\operatorname{lim}_{i\to\infty}x_i=x.
	$$ 
	By \eqref{area estimate} and Lemma \ref{large blowdown}, $x\neq 0$. Given $\delta\in(0,1/2)$, let $\hat \Sigma^\delta_i$ be the connected component of $\hat\Sigma_i\setminus B_{\delta}(0)$ containing $x_i$.  
	Using \eqref{smooth large blowdown curv est}, it follows that $\hat \Sigma_i^\delta$ converges to $S_{1}((1-|x|^{-1})\,x)\setminus B_{\delta}(0)$ in $C^2$.
	In particular,
	$$
	\int_{\hat \Sigma^i_\delta}\bar H^2\,\mathrm{d}\bar\mu\geq 16\,\pi-4\,\pi\,\delta^2-o(1).
	$$
	If $\hat\Sigma_i\setminus B_{2\,\delta}(0)$ has more than one component for infinitely many $i$, we may apply the same argument to the second component to conclude that
	$$
	\liminf_{i\to\infty}\int_{\hat\Sigma_i\setminus B_\delta(0)}\bar H^2\,\mathrm{d}\bar\mu\geq 32\,\pi-8\,\pi\,\delta^2.
	$$ 
	This estimate is incompatible with Lemma \ref{large blowdown}. \\ \indent 
	The assertion now follows from  taking a suitable diagonal subsequence. 
\end{proof}
\begin{prop}  \label{small smooth blowodwn}
	Suppose that $\rho_i=o(\lambda_i)$. A subsequence of $\{\rho_i^{-1}\,\Sigma_i\}_{i=1}^\infty$ converges to a flat plane with unit distance to the origin in $C^2$ locally  in $\mathbb{R}^3$ as $i\to\infty$.	
\end{prop}
\begin{proof} 
	Let $\hat\Sigma_i=\rho_i^{-1}\,\Sigma_i$. By Proposition \ref{curv est prop} and Lemma \ref{basic expansions}, we have	
	\begin{align} \label{rho higher order}
		\bar h(\hat\Sigma_i)=o(1)
	\end{align} 
	locally uniformly in $\mathbb{R}^3$. Let $x_i\in \hat\Sigma_i$ with |$x_i|=1$. Given $r>1$, let $\hat \Sigma^r_i$ be the connected component of $\hat\Sigma_i\cap B_r(0)$ containing $x_i$. Using \eqref{rho higher order}, it follows that, passing to a subsequence, $\hat \Sigma^{2\,r}_i$ converges to a bounded subset of a flat plane with unit distance to the origin in $C^2$. If  $\hat\Sigma_i\cap B_r(0)$ has more than one connected component for infinitely many $i$, then, passing to a further subsequence,
	$\hat\Sigma_i\cap B_{2\,r}(0)$ has  a second component that passes through $B_{r}(0)$ and converges  to a bounded subset of a flat plane. In particular,
	$$
	\liminf_{i\to\infty}r^{-2}\,|\hat\Sigma_i\cap B_{2\,r}(0)|_{\bar g}>4\,\pi.
	$$ 
	Applying \eqref{simon equation} with $x=x_i$ and letting $t\to\infty$, we conclude that
	$$
	\liminf_{i\to\infty}\int_{\hat\Sigma_i}\bar H^2\,\mathrm{d}\bar\mu>16\pi.
	$$ 
	This estimate is incompatible with Lemma \ref{large blowdown}. \\ \indent 
	The assertion now follows from  taking a suitable diagonal subsequence. 
\end{proof} 
\begin{prop}  \label{intermediate smooth blowodwn}
	Suppose that $\rho_i=o(\lambda_i)$ and that $\{x_i\}_{i=1}^\infty$ is a sequence of points $x_i\in \Sigma_i$ with $\rho_i=o(|x_i|)$ and $x_i=o(\lambda_i)$. A subsequence of $|x_i|^{-1}\,\Sigma_i$  converges to a flat plane passing through the origin in $C^2$ locally  in $\mathbb{R}^3\setminus\{0\}$ as $i\to\infty$.	
\end{prop}
\begin{proof} 
	The proof is similar to that of Proposition \ref{small smooth blowodwn}. We omit the formal modifications.
\end{proof} 

\begin{coro}
	There holds	 \label{higher order estimates}
	$$
	|h-\lambda_i^{-1}\, g|_{\Sigma_i}|+|x|\,|\nabla  h|=O((\lambda_i^{-1/2}+\rho_i^{-1})\,|x|^{-1}).
	$$
	Likewise,
	$$
	|\bar h-\lambda_i^{-1}\,\bar g|_{\Sigma_i}|_{\bar g}+|x|\,|\bar\nabla \bar h|_{\bar g}=O((\lambda_i^{-1/2}+\rho_i^{-1})\,|x|^{-1}).
	$$
	
\end{coro}
\begin{proof} 
	Let $x_i\in \Sigma_i$ be such that
	$$
	|x_i|^2\,|(\bar\nabla h)(x_i)|=\sup_{x\in \Sigma_i}|x|^2\,|\bar\nabla h|. 
	$$
	By Lemma \ref{diameter in terms of area radius}, $x_i=O(\lambda_i)$. 
	Using either Proposition \ref{large smooth blowodwn} if $
	\lambda_i=O(|x_i|)$, Proposition \ref{small smooth blowodwn} if $x_i=O(\rho_i)$, or Proposition \ref{intermediate smooth blowodwn} if  $x_i=o(\lambda_i)$ and $\rho_i=o(|x_i|)$, it follows that $|x_i|^{-1}\,(\Sigma_i\cap B_{3\,|x_i|/4}(x_i))$ converges either to a subset of a round sphere or to a bounded subset of a flat plane in $C^2$. In particular, the geometry of   $|x_i|^{-1}\,(\Sigma_i\cap B_{3\,|x_i|/4}(x_i))$ is uniformly bounded.  \\ \indent  	
	By Lemma \ref{basic expansions} and Proposition \ref{curv est prop},
	$$
	\Delta H=	\bar \Delta H+O(|x|^{-1}\,|\bar \nabla^2 H|)+O(|x|^{-2}\,|\bar \nabla H|).
	$$
	In conjunction with the area-constrained Willmore equation \eqref{constrained Willmore equation} and Proposition \ref{curv est prop}, we conclude that
	$$
	\bar \Delta H+O(|x|^{-1}\,|\bar \nabla^2 H|)+O((|x|^{-2}\,|\bar \nabla H|)=O((\lambda_i^{-1/2}+\rho_i^{-1})\,|x|^{-2}\,H).
	$$
	By interior $L^4$-estimates as  in \cite[Theorem 9.11]{GilbargTrudinger} and the Sobolev embedding theorem, 
	\begin{equation} \label{nabla h} 
		\begin{aligned}  
			&|x_i|^{5/2}\,\bigg(\int_{\Sigma_i\cap B_{|x_i|/2}(x_i)}|\bar\nabla^2 H|_{\bar g}^4\,\mathrm{d}\bar\mu\bigg)^{1/4}+|x_i|^2\,|(\bar\nabla H)(x_i)|_{\bar g}
			\\&\qquad=O(|x_i|^{1/2})\,\bigg(\int_{\Sigma_i\cap B_{3\,|x_i|/4}(x_i)}( H-2\,\lambda_i^{-1})^4\,\mathrm{d}\bar\mu\bigg)^{1/4}\\&\qquad\qquad+O((\lambda_i^{-1/2}+\rho_i^{-1})\,|x_i|^{1/2})\,\bigg(\int_{\Sigma_i\cap B_{3\,|x_i|/4}(x_i)}H^4\,\mathrm{d}\bar\mu\bigg)^{1/4}
			\\&\qquad=O(\lambda_i^{-1/2}+\rho_i^{-1}).
		\end{aligned} 
	\end{equation} 
	We have used Proposition \ref{curv est prop} in the last equation. \\ \indent 
	Applying the same argument to \eqref{first simon}, using also \eqref{nabla h}, we conclude that
	$$
	|x_i|^2\,	|(\bar\nabla \hcirc)(x_i)|=O(\lambda_i^{-1/2}+\rho_i^{-1}).
	$$
	The assertion follows from this and Lemma \ref{basic expansions}.
\end{proof} 

\section{Asymptotic analysis of large area-constrained Willmore spheres} \label{section improved curvature estimates}

We assume that $g$ is a Riemannian metric on $\mathbb{R}^3$ such that, as $x\to\infty$,
\begin{align*} 
g=\left(1+|x|^{-1}\right)^4\,\bar g+\sigma\qquad\text{where}\qquad\partial_J\sigma=O(|x|^{-2-|J|})
\end{align*} 
for every multi-index $J$ with $|J|\leq 4$. \\ \indent 
Let $\xi\in\mathbb{R}^3$ and $\lambda>0$.
Given $u\in C^{\infty}(S_{\lambda}(\lambda\,\xi))$, we define the map
\begin{align*}
\Phi^u_{\xi,\lambda}\,:\,S_{\lambda}(\lambda\,\xi)\to \mathbb{R}^3 \qquad \text{given by} \qquad  \Phi^u_{\xi,\lambda}(x)=x+u(x)\, (\lambda^{-1}\,x-\xi). 
\end{align*}
We denote by \begin{align} \label{Sigma def} \Sigma_{\xi,\lambda}(u)=\Phi^u_{\xi,\lambda}(S_{\lambda}(\lambda\,\xi))\end{align} the Euclidean graph of $u$ over $S_{\lambda}(\lambda\,\xi)$. We tacitly identify functions defined on $\Sigma_{\xi,\lambda}(u)$ with functions defined on $S_{\lambda}(\lambda\,\xi)$ by precomposition with $\Phi^u_{\xi,\lambda}$. 
\\
 \indent We consider a sequence $\{\Sigma_i\}_{i=1}^\infty$  of area-constrained Willmore spheres $\Sigma_i\subset \mathbb{R}^3$ with 
\begin{align} 
m_H(\Sigma_i)\geq0, \qquad \quad
\lim_{i\to\infty}\rho(\Sigma_i)=\infty,\qquad\quad 
\label{slow divergence} 
\rho(\Sigma_i)=o(\lambda(\Sigma_i)).
\end{align} 
We assume that, as $i\to\infty$,
\begin{align} \label{s} 
\log\lambda(\Sigma_i)=o(\rho(\Sigma_i)).
\end{align}
\indent 
The goal of this section is to study the shape of $\Sigma_i$ as $i\to\infty$.
 More precisely, we show that $\Sigma_i$ is a graph over a nearby coordinate sphere, provided $i$ is sufficiently large.
\\ \indent 
We abbreviate $\rho_i=\rho(\Sigma_i)$ and $\lambda_i=\lambda(\Sigma_i)$. \\ \indent 
Passing to a subsequence, we may assume that either $\Sigma_i$  encloses $B_2$ for every $i$ or that the bounded region enclosed by $\Sigma_i$ is disjoint from $B_2$ for every $i$. 
Let $x_i\in \Sigma_i\cap S_{\rho_i}(0)$.  Passing to a further subsequence if necessary, we may assume that there is $\xi\in\mathbb{R}^3$ with $|\xi|=1$ such that
\begin{align} \label{xi}
\lim_{i\to\infty} \rho_i^{-1}\,x_i=-\xi.
\end{align} 
\begin{lem} If $\Sigma_i$ encloses $B_2$ for every $i$, the surfaces \label{cmc blowdown} $\lambda_i^{-1}\,\Sigma_i$ converge to $S_1(\xi)$ in $C^1$. If the bounded region enclosed by $\Sigma_i$ is disjoint from $B_2$ for every $i$, the surfaces $\lambda_i^{-1}\,\Sigma_i$ converge to $S_1(-\xi)$ in $C^1$.
\end{lem}
\begin{proof} 
	This is similar to an argument given in \cite{cmc}. We repeat the argument for the reader's convenience. \\ \indent 
	We first assume that    $\Sigma_i$ encloses $B_2$ for every $i$. \\ \indent 
		We may assume that $\xi=e_3$. Let $a_i\in\mathbb{R}^3$ with $|a_i|=1$ and $a_i\perp x_i,\,e_3$. Let $R_i\in SO(3)$ be the unique rotation with $R(a_i)=a_i$ and $R(x_i)=|x_i|\,e_3$.  
	By \eqref{xi}, $\lim_{i\to\infty}R_i=\operatorname{Id}$. \\ \indent   Let $\gamma_i>0$ be the largest radius such that there is a smooth function  $u_i:\{y\in\mathbb{R}^2:|y|\leq \gamma_i\}\to\mathbb{R}$ with 
	\begin{equation*} 
	\begin{aligned}
	\circ\qquad &|(\bar\nabla u_i)(y)|\leq1\qquad\qquad\qquad\qquad\qquad\qquad\qquad\qquad\qquad\qquad \\[-3.5pt]
	\circ\qquad &(y,\rho_i+u_i(y))\in R_i(\Sigma_i)
	\end{aligned} 
	\end{equation*}
	for all $y\in\mathbb{R}^2$ with $|y|\leq \gamma_i$.
	Clearly, $\gamma_i>0$, $(\bar\nabla u_i)(0)=0$, and $u_i(0)=0$. It follows that 
	\begin{align} \label{phi vs y} 
|y|+\rho_i\leq 	3\,|(y,\rho_i+u_i(y))|\leq6\,(|y|+\rho_i)
	\end{align} 
	and
	\begin{align} \label{first h} 
	|(\bar\nabla^2 u_i)(y)|\leq 8\,|\bar h(R_i(\Sigma_i))((y,\rho_i+u_i(y)))|
	\end{align} 
	for every $y\in \mathbb{R}^2$ with $|y|\leq \gamma_i$. Moreover, by Corollary \ref{higher order estimates},
	\begin{equation} \label{improved grad estimate}  
	\begin{aligned} 
	\bar h(R_i(\Sigma_i))=\lambda_i^{-1}\,\bar g|_{R_i(\Sigma_i)}+O(\lambda_i^{-1/2}+\rho_i^{-1}\,|x|^{-1}).
	\end{aligned} 
	\end{equation} 
	Combining \eqref{first h}, \eqref{improved grad estimate}, and \eqref{phi vs y},  we have
	$$
	|(\bar\nabla^2 u_i)|_y|\leq 16\,\lambda_i^{-1}+O((\lambda_i^{-1/2}+\rho_i^{-1})\,(|y|+\rho_i)^{-1}).
	$$
	Integrating and using \eqref{phi vs y}, Lemma \ref{diameter in terms of area radius}, \eqref{slow divergence}, and \eqref{s}, 
	\begin{equation} \label{grad estimate}
	\begin{aligned}  
	|(\bar\nabla u_i)|_y|\leq 16\,|y|\,\lambda_i^{-1}+O(\log(\rho_i^{-1}\,\lambda_i)\,(\lambda_i^{-1/2}+\rho_i^{-1}))=16\,|y|\,\lambda_i^{-1}+o(1).
	\end{aligned} 
\end{equation} 
	It follows that $32\,\gamma_i\geq \lambda_i$ for all $i$ sufficiently large.  \eqref{grad estimate} also shows that, given $\varepsilon>0$, there is $\delta>0$ such that
	$$
	|\bar\nu(R_i(\Sigma_i))-e_3|\leq \varepsilon\qquad\text{on}\qquad \big\{(y,\rho_i+u_i(y)):y\in\mathbb{R}^2\text{ with }\lambda_i^{-1}\,|y|\leq \delta \big\}.
	$$
	According to Proposition \ref{large smooth blowodwn}, $\lambda_i^{-1}\,R_i(\Sigma_i)$ converges to $S_1(\tilde \xi)$ in $C^2$  locally  in $\mathbb{R}^3\setminus\{0\}$ where $\tilde \xi\in\mathbb{R}^3$. 
	The preceding argument shows that $\tilde \xi=\xi$ and that the convergence is in $C^1$ in $\mathbb{R}^3$. 
	\\   \indent This finishes the proof in the case where each $\Sigma_i$ encloses $B_2$. The case where $B_2$ is disjoint from the bounded region enclosed by $\Sigma_i$ for every $i$ requires only formal modifications. 
\end{proof} 
If $\Sigma_i$ encloses $B_2$, we define
$$
\xi_i=(\lambda_i^{-1}-\rho_i^{-1})\,x_i.
$$
If the bounded region enclosed by $\Sigma_i$ is disjoint from $B_2$, we define
$$
\xi_i=(\lambda_i^{-1}+\rho_i^{-1})\,x_i.
$$ 
Note that, in either case,
\begin{align*}
|1-|\xi_i||=\lambda_i^{-1}\,\rho_i\qquad\text{and}\qquad x_i=\lambda_i\,(1-|\xi_i|^{-1})\,\xi_i\in S_{\lambda_i}(\lambda_i\,\xi_i);
\end{align*}
see Figure \ref{acws sphere Figure}.
	\begin{figure}\centering
	\includegraphics[width=0.7\linewidth]{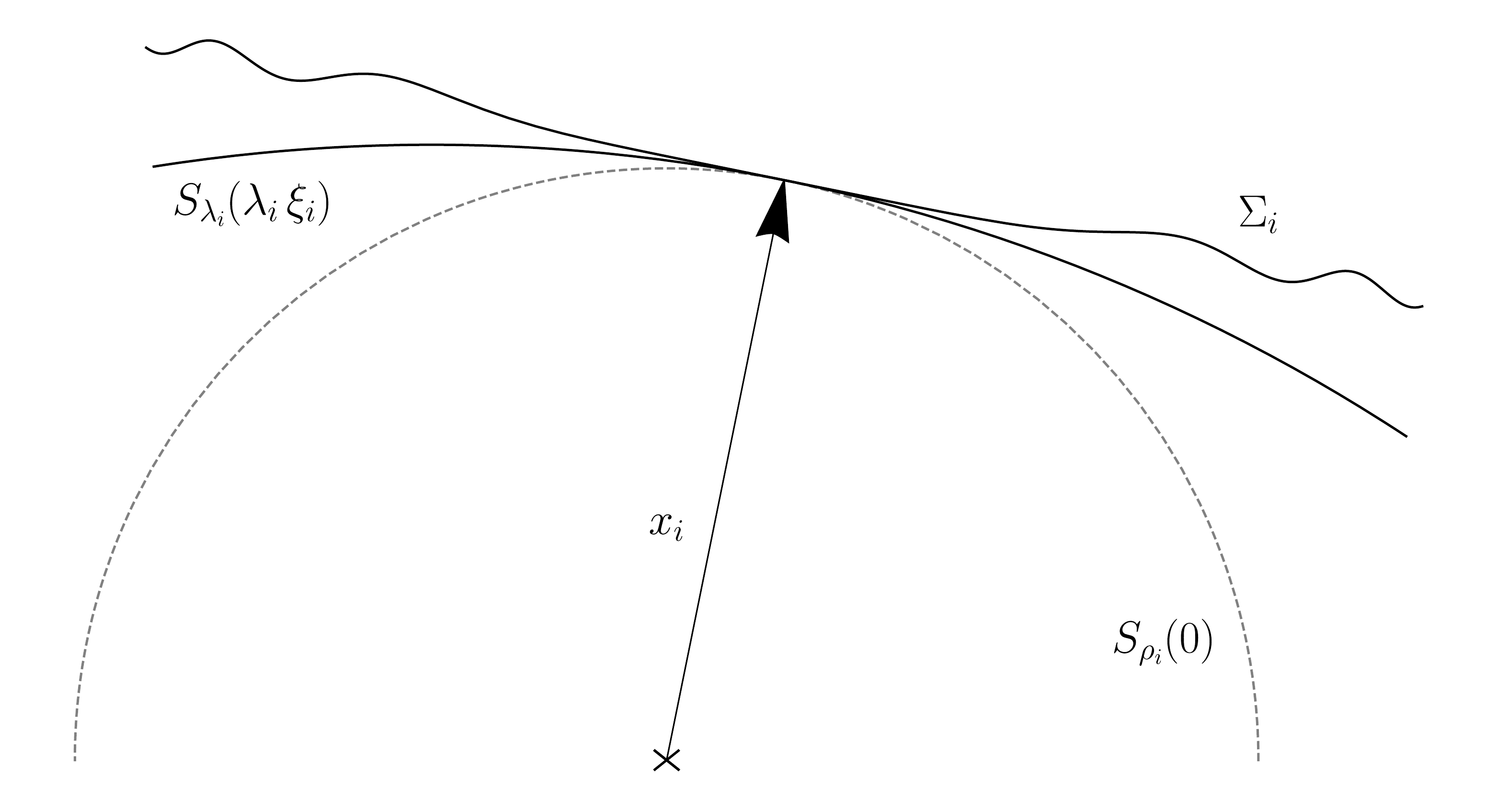}
	\caption{An illustration of $\Sigma_i$ and $S_{\lambda_i}(\lambda_i\,\xi_i)$. The cross marks the origin in the asymptotically flat chart. Here, $\Sigma_i$ encloses $B_2$. }
	\label{acws sphere Figure}
\end{figure}\\
\indent The following lemma is an immediate consequence of Lemma \ref{cmc blowdown}.
\begin{lem} 
	For all  sufficiently large $i$, there are  $u_i\in C^\infty(S_{\lambda_i}(\lambda_i\,\xi_i))$ with the following properties. \label{u initial estimate}  
	\begin{equation*}
	\begin{aligned} 
		&\circ \qquad \Sigma_i=\Sigma_{\xi_i,\lambda_i}(u_i) \\[-2.pt]
		&\circ \qquad u_i(x_i)=0 \\[-2.pt]
		&\circ \qquad (\bar\nabla u_i)(x_i)=0 \\[-2.pt] 
		&\circ \qquad \bar\nabla u_i=o(1)
	\end{aligned} 
	\end{equation*}
\end{lem}
We abbreviate $S_i=S_{\lambda_i}(\lambda_i\,\xi_i)$ and $\Phi_i=\Phi^{u_i}_{\xi_i,\lambda_i}$.
\begin{rema} \label{ui remark} \label{Phi remark}
It follows from Lemma \ref{u initial estimate} that
	$
	u_i(x)=o(|x|)
	$
and
		$
		\Phi_i(x)=x+o(|x|).
		$
\end{rema}
To proceed, we need the following technical lemma.
\begin{lem}
	Let $c\geq 1$ and $\beta:[0,1]\to\mathbb{R}$ be a non-negative, measurable function with
	$$
	\int_0^1 \beta(s)\,\mathrm{d}s\leq \frac1{16}\,c^{-2}\,(1+2\,c)^{-1}\,\exp(-2\,c).
	$$
	Suppose that $\alpha:[0,1]\to\mathbb{R}$ be a differentiable function with absolutely continuous derivative such that \label{ODE lemma} $\alpha(0)=\alpha'(0)=0$ and
	$$
	|\alpha''|\leq c^2\,|\alpha|+c^2\,(\alpha')^2+\beta	
	$$
almost everywhere.	Then
	$$
	|\alpha'|\leq 4\,(1+2\,c)\,\exp(2\,c)\, \int_0^1 \beta(s)\,\mathrm{d}s.
	$$
\end{lem}
\begin{proof}
We will assume that $\beta>0$. The general case follows by approximation. \\ \indent 
	Let $\omega:[0,1]\to\mathbb{R}$ be given by
	$$
	\omega(t)=4\, t\,\exp(2\,c\,t)\,\int_0^1\beta(s)\,\mathrm{d}s.
	$$
	We claim that $|\alpha'(t)|<\omega'(t)$ on $[0,1]$.
To see this, note that 
	\begin{align} \label{w zero} 
	4\,\int_0^1\beta(s)\,\mathrm{d}s<\omega'\leq 4\,(1+2\,c)\,\exp(2\,c)\,\int_0^1\beta(s)\,\mathrm{d}s\leq \frac14\,c^{-2} 
	\end{align} 
	and
	\begin{align} \label{w one}
	0\leq4\,c^2\,\omega<\omega''. 
	\end{align} 
 Suppose that there is $t_0\in(0,1]$ such that $|\alpha'(t_0)|=\omega'(t_0)$ and  $|\alpha'(t)|<\omega'(t)$ on $[0,t_0)$. It follows that $|\alpha(t)|<\omega(t)$ on $[0,t_0)$. Consequently,
 \begin{equation} \label{w two}
	\begin{aligned}  
	\omega'(t_0)=|\alpha'(t_0)|\leq\int_0^{t_0}|\alpha''(s)|\,\text{d}s\leq c^2\,\int_0^{t_0} \omega(s)\,\text{d}s+c^2\,\int_0^{t_0} \omega'(s)^2\,\text{d}s+\int_0^1\,\beta(s)\,\mathrm{d}s.
	\end{aligned} 
	\end{equation} 
By \eqref{w one}, $\omega'(t)\leq \omega'(t_0)$ on $[0,t_0)$.	Using this, \eqref{w zero}, and \eqref{w one}, we have 
	\begin{equation*} 
	\begin{aligned} 
	c^2\,\int_0^{t_0} \omega(s)\,\text{d}s+c^2\,\int_0^t \omega'(s)^2\,\text{d}s\leq \frac14\,\int_0^{t_0}\omega''(s)\,\mathrm{d}s+\frac14\,\int_0^{t_0}\,\omega'(s)\,\mathrm{d}s<\frac12\,\omega'(t_0).
	\end{aligned} 
	\end{equation*} 
In conjunction with \eqref{w two}, we conclude that
	$$
	\omega'(t_0)\leq 2\,\int_0^1\beta(s)\,\mathrm{d}s.
	$$
	This is incompatible with \eqref{w zero}.
	\\ \indent
	It follows that $|\alpha'(t)|<\omega'(t)$ on $[0,1]$. The assertion now follows from \eqref{w zero}.
\end{proof}

\begin{lem}
	There holds \label{u initial estimate 3} 
	\begin{equation*}
	\begin{aligned} 
|x|^{-1}\,	|u_i|+|\bar \nabla u_i|+|x|\,|\bar\nabla^2 u_i|
	= O((\log(\rho_i^{-1}\,\lambda_i))^{1/2}\,(\lambda_i^{-1/2}+\rho_i^{-1})).
	\end{aligned}
	\end{equation*} 
\end{lem}
\begin{proof}
	Let $z_i\in S_i$ and $\gamma_i:[0,1]\to S_i$ be a minimizing geodesic with respect to $\bar g$ such that $\gamma_i(0)=x_i$ and $\gamma_i(1)=z_i$.  Note that
	\begin{align} \label{gamma length estimate} 
	|\dot\gamma_i|\leq \pi\,\lambda_i.
	\end{align} 
Given an integer $\ell$, let
$$
S_{i,\ell}=\{z\in S_i:2^{\ell-1}\,\rho_i\leq |z|<2^\ell\,\rho_i\}.
$$
Note that 
$$
S_i\cap B_{|z|/2}(z)\subset S_{i,\ell-1}\cup S_{i,\ell}\cup  S_{i,\ell+1}
$$
for every $z\in  S_{i,\ell}$ and that
$$
\int_{\gamma\cap S_{i,\ell}}|z|^{-1}\,\mathrm{d}\bar\mu= O(1)
$$	
uniformly for all $\ell$.\\
\indent Let $z\in S_i$. By  Lemma \ref{basic expansions} and Proposition \ref{curv est prop}, 
we have
\begin{align*} 
\bar h(\Sigma_i)(\Phi_i(z))&=h(\Sigma_i)(\Phi_i(z))+O(|(\Phi_i(z))|^{-1}\,|h(\Sigma_i)(\Phi_i(z))|))+O(|\Phi_i(z)|^{-2}) 
\\&=h(\Sigma_i)(\Phi_i(z))+O(|\Phi_i(z)|^{-2})
\end{align*} 
and 
$$
\bar g|_{\Sigma_i}=g|_{\Sigma_i}+O(|\Phi_i(z)|^{-1}).
$$
Using Proposition \ref{int curv estimate} and \eqref{kappa est}, we have
\begin{align*} 
&(h(\Sigma_i)-\lambda_i^{-1}\, g|_{\Sigma_i})(\Phi_i(z))\\&\qquad =O(|\Phi_i(z)|^{-1})\,\bigg(\int_{\Sigma_i\cap B_{|\Phi_i(z)|/4}(\Phi_i(z))}|h-\lambda_i^{-1}\,g|_{\Sigma_i}|^2\,\text{d}\mu\bigg)^{1/2}\\&\qquad\qquad +O(|\Phi_i(z)|^{-2})+O((\lambda_i^{-1/2}+\rho_i^{-1})\,\lambda_i^{-1}).
\end{align*} 
Using Remark \ref{ui remark} and  Lemma \ref{graphical geometric components}, we conclude that 
\begin{equation}  \label{nabla2 eq} 
	\begin{aligned}
(\bar\nabla^2 u_i)(z)=\,& O(|z|^{-1})\,\bigg(\int_{S_i\cap B_{|z|/2}(z)}|h(\Sigma_i)-\lambda_i^{-1}\,g|_{\Sigma_i}|^2\,\text{d}\mu\bigg)^{1/2} \\&\qquad+O(|z|^{-2})+ O((\lambda_i^{-1/2}+\rho_i^{-1})\,\lambda_i^{-1})
  +O(\lambda_i^{-2}\,|u_i(z)|)+O(\lambda_i^{-1}\,|\bar\nabla u_i(z)|^2).
\end{aligned} 
\end{equation} 
\indent We have
\begin{equation} \label{beta one}
\begin{aligned}   
\int_{\gamma}|z|^{-2}+\int_{\gamma}(\lambda_i^{-1/2}+\rho_i^{-1})\,\lambda_i^{-1}=O(\rho_i^{-1})+O(\lambda_i^{-1/2}+\rho_i^{-1})=O(\lambda_i^{-1/2}+\rho_i^{-1}). 
\end{aligned} 
\end{equation} 
Let
$
k_i=\lceil (\log2)^{-1}\,\log(\rho_i^{-1}\,|z_i|)\rceil
$
and note that $k_i=O(\log(\rho_i^{-1}\,\lambda_i))$. We have
\begin{equation} \label{beta two} 
\begin{aligned}
&\int_{\gamma}|z|^{-1}\,\bigg(\int_{S_i\cap B_{|z|/2}(z)}|h(\Sigma_i)-\lambda_i^{-1}\,g|_{\Sigma_i}|^2\,\text{d}\mu\bigg)^{1/2}\\
&\qquad = O(1)\,\sum_{\ell=1}^{k_i}\int_{\gamma\cap S_{i,\ell}}|z|^{-1}\,\mathrm{d}\bar\mu(z)\,\bigg(\int_{S_{i,\ell-1}\cup S_{i,\ell}\cup S_{i,\ell+1}}|h(\Sigma_i)-\lambda_i^{-1}\,g|_{\Sigma_i}|^2\,\text{d}\mu\bigg)^{1/2}\\ 
&\qquad =O(\sqrt{k_i})\,\bigg(\sum_{\ell=1}^{k_i}\,\int_{S_{i,\ell-1}\cup S_{i,\ell}\cup S_{i,\ell+1}}|h(\Sigma_i)-\lambda_i^{-1}\,g|_{\Sigma_i}|^2\,\text{d}\mu\bigg)^{1/2}
\\ 
&\qquad =O((\log(\rho_i^{-1}\,\lambda_i))^{1/2})\,\bigg(\int_{\Sigma_i}|h(\Sigma_i)-\lambda_i^{-1}\,g|_{\Sigma_i}|^2\,\text{d}\mu\bigg)^{1/2}
\\&\qquad=O((\log(\rho_i^{-1}\,\lambda_i))^{1/2}\,(\lambda_i^{-1/2}+\rho_i^{-1})). 
\end{aligned}
\end{equation} 
We have used Lemma \ref{hcirc int etsimate} in the last equation.
\\ \indent 
Let $c\geq 1$ and $\alpha_i,\,\beta_i:[0,1]\to\mathbb{R}$ be given by
$$
\alpha_i(s)=\int_0^s|(\bar\nabla u_i)(\gamma(t))|\,\mathrm{d}t
$$
and  
\begin{align*} 
\beta_i(s)&=c\,|\dot \gamma_i(s)|\,|\gamma_i(s)|^{-1}\,\bigg(\int_{S_i\cap B_{|\gamma_i(s)|/2}(\gamma_i(s))}|h(\Sigma_i)-\lambda_i^{-1}\,g|_{\Sigma_i}|^2\,\text{d}\mu\bigg)^{1/2} \\&\qquad  +c\,|\dot \gamma_i(s)|\,|\gamma_i(s)|^{-2}+c\,|\dot \gamma_i(s)|\,(\lambda_i^{-1/2}+\rho_i^{-1})\,\lambda_i^{-1}.
\end{align*} 
By Lemma \ref{u initial estimate},
$
|u_i(\gamma_i(s))|\leq |\dot\gamma_i(s)|\,\alpha_i(s)
$
for all $s\in[0,1]$.
Moreover, whenever $\alpha_i''(s)$ exists, there holds $|\alpha_i''(s)|\leq |\dot\gamma_i(s)|\,|(\bar\nabla ^2u_i)(\gamma_i(s))|$. In conjunction with \eqref{gamma length estimate} and \eqref{nabla2 eq}, we obtain
$$
|\alpha_i''|\leq c^2\,|\alpha_i|+c^2\,(\alpha_i')^2+\beta_i$$
almost everywhere for all large $i$ provided that $c\geq 1$ is sufficiently large. Clearly, $\alpha_i(0)=0$ and, by Lemma \ref{u initial estimate}, $\alpha_i'(0)=0$.
Moreover, using \eqref{beta one} and \eqref{beta two}, we have
$$
\int_0^1 \beta(s)\,\mathrm{d}s=o(1).
$$
 Using Lemma \ref{ODE lemma} and (\ref{nabla2 eq}-\ref{beta two}), we obtain 
$$
|(\bar\nabla u_i)(z_i)|=\alpha_i'(1)=O((\log(\rho_i^{-1}\,\lambda_i))^{1/2}\,(\lambda_i^{-1/2}+\rho_i^{-1})).
$$
Integrating and using Lemma \ref{u initial estimate}, we have 
$$
|x|^{-1}\,u_i=O((\log(\rho_i^{-1}\,\lambda_i))^{1/2}\,(\lambda_i^{-1/2}+\rho_i^{-1})).
$$
Returning to \eqref{nabla2 eq} and using Lemma \ref{hcirc int etsimate}, we have  
$$
|x|\,\bar\nabla ^2u_i=O((\log(\rho_i^{-1}\,\lambda_i))^{1/2}\,(\lambda_i^{-1/2}+\rho_i^{-1})).
$$
\indent The assertion follows.
\end{proof} 

\section{Asymptotic analysis of the mean curvature} \label{section mean curvature estimates}
We assume that $g$ is a Riemannian metric on $\mathbb{R}^3$ such that, as $x\to\infty$,
\begin{align} \label{asymptotic to Schwarzschild 1.5}
g=\left(1+|x|^{-1}\right)^4\,\bar g+\sigma\qquad\text{where}\qquad\partial_J\sigma=O(|x|^{-2-|J|})
\end{align} 
for every multi-index $J$ with $|J|\leq 4$. Let $\{\Sigma_i\}_{i=1}^\infty$ be a sequence of area-constrained Willmore spheres $\Sigma_i\subset \mathbb{R}^3$ with 
\begin{align} \label{slow divergence 1.5}  
m_H(\Sigma_i)\geq0, \qquad \quad
\lim_{i\to\infty}\rho(\Sigma_i)=\infty,\qquad\quad 
\rho(\Sigma_i)=o(\lambda(\Sigma_i))
\end{align} 
and assume that, as $i\to\infty$,
\begin{align} \label{s 1.5} 
\log\lambda(\Sigma_i)=o(\rho(\Sigma_i)).
\end{align} 
\indent As before, we abbreviate $\lambda_i=\lambda(\Sigma_i)$ and $\rho_i=\rho(\Sigma_i)$. Recall from Lemma \ref{u initial estimate} that, for all $i$ large, $\Sigma_i=\Sigma_{\xi_i,\lambda_i}(u_i)$ is the Euclidean graph of a function $u_i$ over the sphere $S_i=S_{\lambda_i}(\lambda_i\,\xi_i)$. Moreover, recall that $\Phi_i=\Phi^{u_i}_{\xi_i,\lambda_i}$ and that we identify functions defined on $\Sigma_i$ with functions defined on $S_i$ by precomposition with $\Phi_i$.
\\ \indent
The goal of this section is to  obtain an improved estimate for the mean curvature $H(\Sigma_i)$. To this end, we analyze the area-constrained Willmore equation \eqref{constrained Willmore equation}. \\ \indent 
Given  $\xi\in\mathbb{R}^3$ and $\lambda>0$, let $\Lambda_0(S_\lambda(\lambda\,\xi))$ be the space of constant functions on $S_{\lambda}(\lambda\,\xi)$ and $\Lambda_0(S_\lambda(\lambda\,\xi))^\perp$  its orthogonal complement in $C^\infty(S_{\lambda}(\lambda\,\xi))$ with respect to the Euclidean $L^2$-inner product. We use $\operatorname{proj}_{\Lambda_0(S_\lambda(\lambda\,\xi))}$ and $\operatorname{proj}_{\Lambda_0(S_\lambda(\lambda\,\xi))^\perp}$ to denote the $L^2$-projections onto these spaces. \\ \indent We need the following gradient estimate for the Laplace operator.  
\begin{lem}
	There is a constant $c>0$ with the following property. Let $\xi\in\mathbb{R}^3$  and $\lambda>0$. Suppose that $u,\,f\in\Lambda_0(S_\lambda(\lambda\,\xi))^\perp$ are such that \label{PDE lemma 2}
	$
	\bar\Delta u=f.
	$
	Then 
			\begin{align*} 
	\sup_{x\in S_{\lambda}(\lambda\,\xi)}|x|\,|\bar\nabla u(x)|\leq c\,\bigg(\int_{S_\lambda(\lambda\,\xi)}|f|\,\mathrm{d}\bar\mu+\sup_{x\in S_{\lambda}(\lambda\,\xi)}\,|x|^2\,|f|\bigg).
	\end{align*} 
\end{lem}
\begin{proof}
			By scaling, we may assume that $\lambda=1$ and
				$$
			\int_{S_1(\xi)}|f|\,\mathrm{d}\bar\mu+\sup_{x\in S_1(\xi)}|x|^2\,|f|=1.
			$$  
\indent Recall from, e.g.,~\cite[\S A.1]{Beltran}, that the Green's function of  $\bar\Delta:\Lambda_0(S_1(0))^\perp\to \Lambda_0(S_1(0))^\perp$ is given by
	$$
	G(x,y)=\frac{1}{2\,\pi}\,\log|x-y|.
	$$
	It follows that
	\begin{align} 
	 (\bar\nabla u)(x)=\,&\int_{S_1(\xi)}(\bar\nabla G)(x,y)\,f(y)\,\mathrm{d}\bar\mu(y) \label{tocompute2}
	\end{align} 
	where differentiation is with respect to $x$. 
	\begin{figure}\centering
		\includegraphics[width=0.4\linewidth]{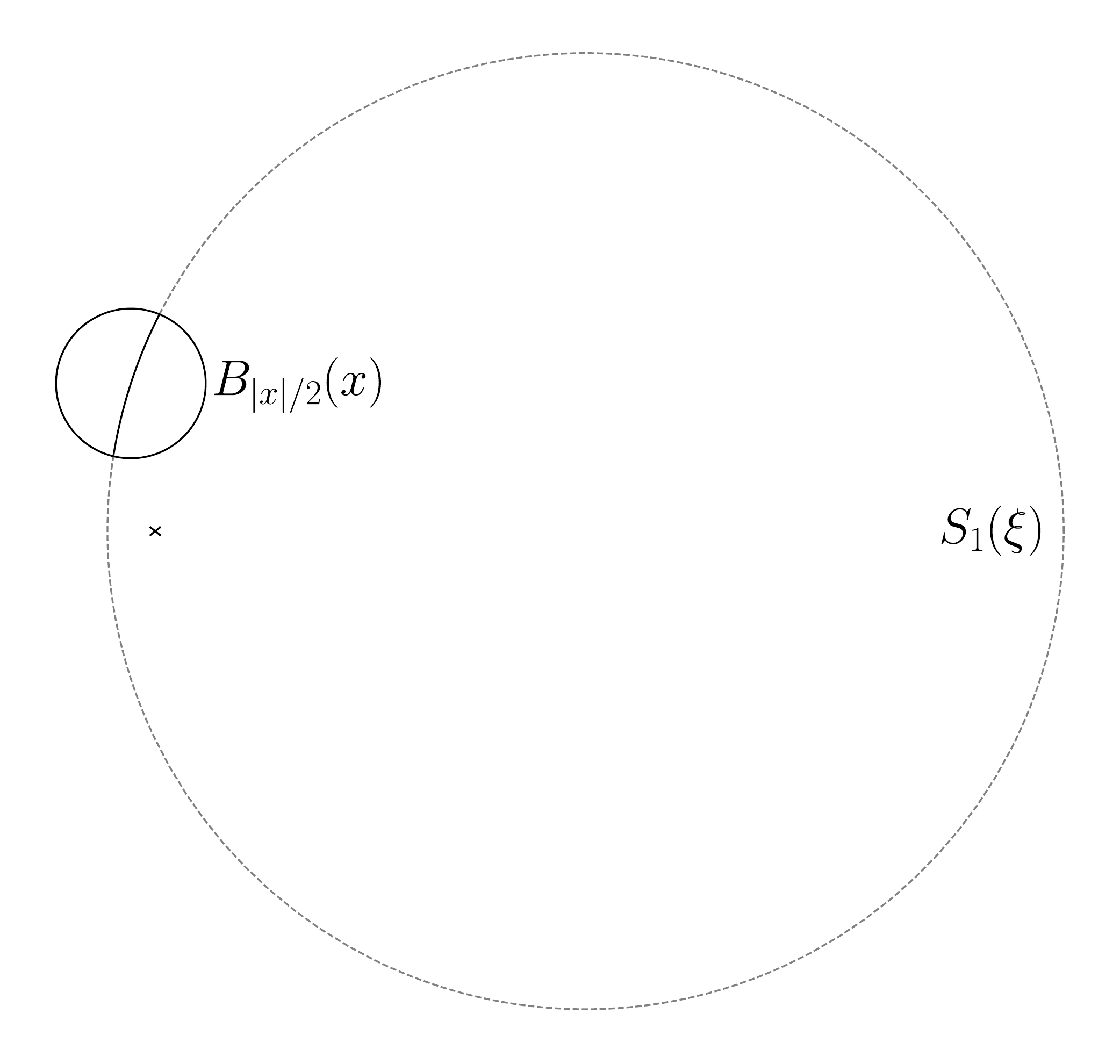}
		\caption{An illustration of the partition \eqref{illustration} for $|\xi|\approx 1$ and $|x|\approx |1-\xi|$. The cross marks the origin of $\mathbb{R}^3$. The gradient of the Green's function $G$ is large within the part of $S_1(\xi)$ illustrated by the black line while $f$ may be large within the part of $S_1(\xi)$ illustrated by the dashed, gray line.  }
		\label{pde lemma figure}
	\end{figure}
Note that 
	$$
	(\bar \nabla G)(x,y)= O(1)\,|x-y|^{-1}
	$$
	for all $x,y\in\mathbb{R}^3$ with $x\neq y$.
	\\ \indent 
	Let $x\in S_1(\xi)$. We may assume that $x\neq 0$.
	 We estimate the integral \eqref{tocompute2} over the regions
	\begin{align} \label{illustration}  
	\{y\in S_1(\xi)\,:\,2\,|y-x|\geq|x|\}\qquad\text{and}\qquad \{y\in S_1(\xi)\,:\,2\,|y-x|\leq|x|\}
	\end{align} 
	separately;	see Figure \ref{pde lemma figure}. 	We have
	\begin{align*}
	\int_{\{y\in S_1(\xi)\,:\,2\,|y-x|\geq|x|\}}\,|x-y|^{-1}\,|f|\,\mathrm{d}\bar\mu(y)
	\leq O(1)\,|x|^{-1}\,\int_{S_1(\xi)}|f|\,\mathrm{d}\bar\mu(y)
	 \leq  O(1)\,|x|^{-1}.
	\end{align*}
	Likewise, 
	\begin{align*}
	\int_{\{y\in S_1(\xi)\,:\,2\,|y-x|\leq|x|\}}\,|x-y|^{-1}\,|f|\,\mathrm{d}\bar\mu(y)
	& \leq \int_{\{y\in S_1(\xi)\,:\,2\,|y-x|\leq|x|\}}\,|x-y|^{-1}\,|y|^{-2}\,\mathrm{d}\bar\mu(y)
	\\&\leq O(1)\,|x|^{-2}\,\int_{\{y\in S_1(\xi)\,:\,2\,|y-x|\leq|x|\}}|x-y|^{-1}\,\mathrm{d}\bar \mu(y)
	\\&\leq O(1)\,|x|^{-1}.
	\end{align*}
\indent 	The assertion follows from these estimates.
\end{proof}
\begin{rema}  Let $\xi\in\mathbb{R}^3$ with $1/2<|\xi|<1$ or $1<|\xi|<3/2$ and $\lambda>0$. \label{green remark}
	By Lemma \ref{integral lemma}, there is a constant $c>0$ with the following properties.
			\begin{align*} 
	&\circ\qquad 	\int_{S_\lambda(\lambda\,\xi)}|f|\,\mathrm{d}\bar\mu\leq c\,\lambda^2\,\sup_{x\in S_{\lambda}(\lambda\,\xi)}\,|f|\\[-2.pt]
	&\circ\qquad \int_{S_\lambda(\lambda\,\xi)}|f|\,\mathrm{d}\bar\mu\leq c\,|\log|1-|\xi|||\,\sup_{x\in S_{\lambda}(\lambda\,\xi)}\,|x|^2\,|f|\\[-2.pt]
	&\circ\qquad 	\int_{S_\lambda(\lambda\,\xi)}|f|\,\mathrm{d}\bar\mu\leq c\,|1-|\xi||^{-1}\,\lambda^{-1}\,\sup_{x\in S_{\lambda}(\lambda\,\xi)}\,|x|^3\,|f|
	\end{align*}
\end{rema}
 \begin{lem}
As $i\to\infty$, there holds
$$
\operatorname{proj}_{\Lambda_0(S_i)}H(\Sigma_i)=\frac{1}{4\,\pi}\,\lambda_i^{-2}\,\int_{S_i}H(\Sigma_i)\,\mathrm{d}\bar\mu=2\,\lambda_i^{-1}+o(\lambda_i^{-1})
$$
and
\begin{align*}
\kappa(\Sigma_i)=o(\lambda_i^{-2}\,\rho_i^{-1}).
\end{align*} 
Moreover, \label{improved H estimate}
\begin{align*}
\operatorname{proj}_{\Lambda_0(S_i)^\perp}H(\Sigma_i)=-4\,\lambda_i^{-1}\,|x|^{-1}+o(\lambda_i^{-1}\,\rho_i^{-1}).
\end{align*}
\end{lem}
\begin{proof}
	By Proposition \ref{curv est prop} and Remark \ref{Phi remark},
	\begin{align} \label{H initial}  
	H(\Sigma_i)=2\,\lambda_i^{-1}+O((\lambda_i^{-1/2}+\rho_i^{-1})\,|x|^{-1}).
	\end{align} 
Using Lemma \ref{integral lemma}, we have
	\begin{equation*}
	\begin{aligned} 
	\int_{S_i} H(\Sigma_i)\,\mathrm{d}\bar\mu=8\,\pi\,\lambda_i+O(\lambda_i^{-1/2}+\rho_i^{-1})\,\int_{S_i}|x|^{-1}\,\mathrm{d}\bar\mu=8\,\pi\,\lambda_i+O((\lambda_i^{-1/2}+\rho_i^{-1})\,\lambda_i)
	\end{aligned} 
	\end{equation*} 
and 
\begin{align} \label{h L1 est} 
	\int_{S_i}|H(\Sigma_i)|\,\mathrm{d}\bar\mu=O(\lambda_i).
\end{align}
	Consequently,
	\begin{align} \label{proj 0 He stimate}
	\operatorname{proj}_{\Lambda_0(S_i)}H(\Sigma_i)=2\,\lambda_i^{-1}+O((\lambda_i^{-1/2}+\rho_i^{-1})\,\lambda_i^{-1}).
	\end{align}
	\indent 
	We define the function $F_i:\Sigma_i\to\mathbb{R}$ by $F_i=N^{-1}\,H(\Sigma_i)$ where 
	$$
	N:\mathbb{R}^3\setminus\{0\}\to\mathbb{R} \qquad\text{is given by}\qquad N(x)=(1+|x|^{-1})^{-1}\,(1-|x|^{-1})
	$$
	is the potential function of Schwarzschild; see \eqref{potential function}.
	By  Remark \ref{Phi remark},
	\begin{equation} \label{derivative comparison}
	\begin{aligned}
	F_i=\,&(1+2\,|x|^{-1}+o(|x|^{-1}))\,H(\Sigma_i).
	\end{aligned}
	\end{equation}
It follows that 
$$
\int_{S_i}F_i\,\mathrm{d}\bar\mu=\int_{S_i} H(\Sigma_i)\,\mathrm{d}\bar\mu+o(1)\,\int_{S_i}|H(\Sigma_i)|\,\mathrm{d}\bar\mu=\int_{S_i}H(\Sigma_i)\,\mathrm{d}\bar\mu+o(\lambda_i).
$$
In conjunction with \eqref{proj 0 He stimate}, we obtain
\begin{align} \label{proj F_i} 
\operatorname{proj}_{\Lambda_0(S_i)}F_i=2\,\lambda_i^{-1}+o(\lambda_i^{-1}).
\end{align} 
By Lemma \ref{support function lemma}, Lemma \ref{ambient expansions}, Lemma \ref{basic expansions}, and  Proposition \ref{curv est prop},  we have
\begin{equation} \label{- one} 
	\begin{aligned} 
		\Delta_{\Sigma_i} F_i=\,&-(|\hcirc(\Sigma_i)|^2+\kappa(\Sigma_i))F_i\\&\qquad+O(|x|^{-4}+\lambda_i^{-1}\,|x|^{-2}+(\lambda_i^{-1/2}+\rho_i^{-1})\,|x|^{-3})\,|F_i|
		 +O(|x|^{-3})\,|x|\,|\nabla F_i|.
	\end{aligned}
\end{equation} 
Using Lemma \ref{basic expansions}, Proposition \ref{curv est prop}, and Remark \ref{Phi remark}, we have
$$
\bar \Delta_{\Sigma_i}F_i=(1+O(|x|^{-1}))\,\Delta_{\Sigma_i}F_i+O(|x|^{-3})\,(|x|\,|\bar\nabla F_i|+|x|^2\,|\bar\nabla ^2F_i|).
$$
In conjunction with \eqref{- one} and Proposition \ref{curv est prop}, we conclude that
\begin{equation} \label{zero} 
\begin{aligned} 
\bar\Delta_{\Sigma_i}F_i=\,&-(|\hcirc(\Sigma_i)|^2+\kappa(\Sigma_i))\,F_i
\\&\qquad +O(\lambda_i^{-1}\,|x|^{-2}+(\lambda_i^{-1/2}+\rho_i^{-1})\,|x|^{-3})\,|F_i|
 +O(|x|^{-3})\,(|x|\,|\bar\nabla F_i|+|x|^2\,|\bar\nabla^2 F_i|).
\end{aligned} 
\end{equation}
Using Lemma \ref{second derivative estimate},  Lemma \ref{u initial estimate 3}, and Lemma \ref{diameter in terms of area radius}, we have
\begin{equation*}
\begin{aligned} 
\bar \Delta_{\Sigma_i}F_i=(1+o(1))\,	\bar \Delta_{S_i}F_i
   +O(\log(\rho_i^{-1}\,\lambda_i)\,(\lambda_i^{-1/2}+\rho_i^{-1})^2\,|x|^{-2})\,(|x|\,|\bar \nabla F_i|+|x|^2\,|\bar \nabla^2 F_i|).
\end{aligned} 
\end{equation*}	
In conjunction with \eqref{zero}, \eqref{slow divergence 1.5}, and \eqref{s 1.5}, we  obtain
\begin{equation}
\label{second}  
\begin{aligned} 
&\bar\Delta_{S_i}F_i=O(|\hcirc(\Sigma_i)|^2+|\kappa(\Sigma_i)|)\,|F_i|
\\&\qquad\qquad \qquad +O(\lambda_i^{-1}\,|x|^{-2}+(\lambda_i^{-1/2}+\rho_i^{-1})\,|x|^{-3})\,|F_i|
\\&\qquad\qquad\qquad  +O((\lambda_i^{-1/2}+\rho_i^{-1})\,|x|^{-2})\,(|x|\,|\bar \nabla F_i|+|x|^2\,|\bar \nabla^2 F_i|).
\end{aligned} 
\end{equation}
\indent
According to Proposition \ref{curv est prop} and Remark \ref{Phi remark}, 
\begin{align*} 
\sup_{x\in S_i}|x|^2\,(|\hcirc(\Sigma_i)|^2+|\kappa(\Sigma_i)|+\lambda_i^{-1}\,|x|^{-2}+(\lambda_i^{{-1/2}}+\rho_i^{-1})\,|x|^{-3}) =O((\lambda_i^{-1/2}+\rho_i^{-1})^2+\lambda_i^2\,|\kappa(\Sigma_i)|).
\end{align*} 
Using Lemma \ref{basic expansions}, Remark \ref{Phi remark}, and Lemma \ref{u initial estimate}, we have
\begin{align} \label{area element} 
\mathrm{d}\mu(\Sigma_i)=(1+o(1))\,\mathrm{d}\bar\mu(\Sigma_i)=(1+o(1))\,\mathrm{d}\bar\mu(S_i).
\end{align} 
In conjunction with Lemma \ref{hcirc int etsimate} and Lemma \ref{integral lemma}, we see that 
\begin{align*} 
&\int_{S_i}(|\hcirc(\Sigma_i)|^2+|\kappa(\Sigma_i)|+\lambda_i^{-1}\,|x|^{-2}+(\lambda_i^{{-1/2}}+\rho_i^{-1})\,|x|^{-3})\,\mathrm{d}\bar\mu\\&\qquad =O((\lambda_i^{-1/2}+\rho_i^{-1})^2+\lambda_i^2\,|\kappa(\Sigma_i)|+\log(\rho_i^{-1}\,\lambda_i)\,\lambda_i^{-1}).
\end{align*} 
Likewise, 
$$
\int_{S_i}(\lambda_i^{-1/2}+\rho_i^{-1})\,|x|^{-2}\,\mathrm{d}\bar\mu=O(\log(\rho_i^{-1}\,\lambda_i)\,(\lambda_i^{-1/2}+\rho_i^{-1})).
$$
 Using \eqref{second} and  Lemma \ref{PDE lemma 2},  we conclude that
 \begin{equation} \label{optimal}
 \begin{aligned} 
 \sup_{x\in S_i}|x|\,|\bar\nabla F_i|&=O((\lambda_i^{-1/2}+\rho_i^{-1})^2+\lambda_i^2\,|\kappa(\Sigma_i)|+\log(\rho_i^{-1}\,\lambda_i)\,\lambda_i^{-1})\,\sup_{x\in S_i} |F_i|
 \\&\qquad +O(\log(\rho_i^{-1}\,\lambda_i)\,(\lambda_i^{-1/2}+\rho_i^{-1}))\,(\operatorname{sup}_{x\in S_i}|x|\,|\bar\nabla F_i|+\operatorname{sup}_{x\in S_i}|x|^2\,|\bar\nabla^2 F_i|).
 \end{aligned} 
 \end{equation}
By standard elliptic theory, 
\begin{equation}\label{higher order} 
 \begin{aligned} 
 \sup_{x\in S_i}|x|^2\,|\bar\nabla ^2 F_i|=O(1)\,\sup_{x\in S_i}|x|\,|\bar\nabla F_i|+O((\lambda_i^{-1/2}+\rho_i^{-1})^2+\lambda_i^2\,|\kappa(\Sigma_i)|)\,\sup_{x\in S_i}|F_i|;
 \end{aligned}
\end{equation} 
 see Remark \ref{elliptic estimate}.
 Using \eqref{slow divergence 1.5} and \eqref{s 1.5} and absorbing, we conclude that 
  \begin{equation} \label{three} 
 \begin{aligned} 
 \sup_{x\in S_i}|x|\,|\bar\nabla F_i|+\sup_{x\in S_i}|x|^2\,|\bar\nabla F_i|^2
 =O((\lambda_i^{-1/2}+\rho_i^{-1})^2+\lambda_i^2\,|\kappa(\Sigma_i)|+\log(\rho_i^{-1}\,\lambda_i)\,\lambda_i^{-1})\,\sup_{x\in S_i} |F_i|.
 \end{aligned} 
 \end{equation}
 Note that there is $z\in S_i$ with $(\operatorname{proj}_{\Lambda_0(S_i)^\perp}F_i)(z)=0$. Integrating, we find
 $$
 \sup_{x\in S_i}|\operatorname{proj}_{\Lambda_0(S_i)^\perp}F_i|=O(\log(\rho_i^{-1}\,\lambda_i))\,\sup_{x\in S_i}|x|\,|\bar\nabla F_i|.
 $$
 Using \eqref{kappa est}, \eqref{slow divergence 1.5}, and \eqref{s 1.5}, we have
 $$
 (\lambda_i^{-1/2}\,+\rho_i^{-1})^2+\lambda_i^2\,|\kappa(\Sigma_i)|+\log(\rho_i^{-1}\,\lambda_i)\,\lambda_i^{-1}=o((\log(\rho_i^{-1}\,\lambda_i))^{-1}).
 $$
 Returning to \eqref{three}  and absorbing,  we obtain
 \begin{equation}  \label{four}
  \begin{aligned} 
 &(\log(\rho_i^{-1}\,\lambda_i))^{-1}\,\sup_{x\in S_i}|\operatorname{proj}_{\Lambda_0(S_i)^\perp}F_i|+\sup_{x\in S_i}|x|\,|\bar\nabla F_i|+\sup_{x\in S_i}|x|^2\,|\bar\nabla F_i|^2
 \\&\qquad =O((\lambda_i^{-1/2}\,+\rho_i^{-1})^2+\lambda_i^2\,|\kappa(\Sigma_i)|+\log(\rho_i^{-1}\,\lambda_i)\,\lambda_i^{-1})\,|\operatorname{proj}_{\Lambda_0(S_i)}F_i|.
 \\&\qquad =O((\lambda_i^{-1/2}\,+\rho_i^{-1})^2+\lambda_i^2\,|\kappa(\Sigma_i)|+\log(\rho_i^{-1}\,\lambda_i)\,\lambda_i^{-1})\,\lambda_i^{-1}.
 \end{aligned} 
 \end{equation}
 We have used \eqref{proj F_i} in the last equation.
 In particular, using \eqref{kappa est},
 \begin{align} \label{five}
 F_i=\operatorname{proj}_{\Lambda_0(S_i)}F_i+\operatorname{proj}_{\Lambda_0(S_i)^\perp}F_i=O(\lambda_i^{-1}).
 \end{align} 
 Using \eqref{zero}, \eqref{four}, \eqref{five}, \eqref{slow divergence 1.5}, and \eqref{s 1.5}, we have
 \begin{align*} 
 \bar\Delta_{\Sigma_i}F_i&=-\kappa(\Sigma_i)\,F_i+O(\lambda_i\,|x|^{-3}\,|\kappa(\Sigma_i)|)+O(\lambda_i^{-1}\,|\hcirc(\Sigma_i)|^2)\\&\qquad+O(\lambda_i^{-2}\,|x|^{-2})+O((\lambda_i^{-1/2}+\rho_i^{-1})\,\lambda_i^{-1}\,|x|^{-3}).
 \end{align*} 
Integrating and using \eqref{proj F_i}, \eqref{area element},  Lemma \ref{integral lemma}, and Lemma \ref{hcirc int etsimate}, we obtain
\begin{align} \label{kappa} 
\kappa(\Sigma_i)=O(\lambda_i^{-2}\,\rho_i^{-2})+O(\log(\rho_i^{-1}\,\lambda_i)\,\lambda_i^{-3}).
\end{align}
Returning to \eqref{four}, we conclude that
\begin{align} \label{proj 0} 
\operatorname{proj}_{\Lambda_0(S_i)^\perp }F_i=o(\lambda_i^{-1}\,\rho_i^{-1}).
\end{align} 
Moreover, by \eqref{kappa}, $\kappa(\Sigma_i)=o(\lambda_i^{-2}\,\rho_i^{-1})$. 
We have used \eqref{slow divergence 1.5} and \eqref{s 1.5} in both of these estimates.\\
\indent
 By \eqref{derivative comparison}, \eqref{five}, and \eqref{proj 0}, we have
\begin{align*} 
\operatorname{proj}_{\Lambda_0(S_i)^{\perp}}H(\Sigma_i)=\,&\operatorname{proj}_{\Lambda_0(S_i)^{\perp}}F_i+\operatorname{proj}_{\Lambda_0(S_i)^{\perp}}O(\rho_i^{-1})\,H(\Sigma_i)\\=\,&o(\lambda_i^{-1}\,\rho_i^{-1})+O(\rho_i^{-1})\,\sup_{x\in S_i}|H(\Sigma_i)|
\\=\,&o(\lambda_i^{-1}\,\rho_i^{-1})+O(\rho_i^{-1})\,\sup_{x\in S_i}|F_i|
\\=\,&o(\lambda_i^{-1}).
\end{align*} 
In conjunction with \eqref{proj 0 He stimate} and \eqref{derivative comparison}, we obtain
$$
F_i=H(\Sigma_i)+4\,|x|^{-1}\,\lambda_i^{-1}+o(\rho_i^{-1}\,\lambda_i^{-1}).
$$
By Lemma \ref{integral lemma},
$$
\operatorname{proj}_{\Lambda_0(S_i)}|x|^{-1}=O(\lambda_i^{-1})=o(\rho_i^{-1}).
$$
Using \eqref{proj 0}, we conclude that
$$
\operatorname{proj}_{\Lambda_0(S_i)^\perp}H(\Sigma_i)=-4\,|x|^{-1}\,\lambda_i^{-1}+o(\lambda_i^{-1}\,\rho_i^{-1}).
$$

\indent The assertion follows.
\end{proof}
\begin{rema} \label{elliptic estimate}
	We provide additional details on how to obtain \eqref{higher order}. \\ \indent  Let $z_i\in S_i$ and $a_i\in T_{z_i}S_i$ with $|a_i|=1$. The estimates below are independent of these choices. \\	\indent Let $X_i=a_i^\top$. By interior $L^4$-estimates as  in \cite[Theorem 9.11]{GilbargTrudinger} and the Sobolev embedding theorem, using also Lemma \ref{basic expansions}, Lemma \ref{second derivative estimate}, Proposition \ref{curv est prop}, and Lemma \ref{u initial estimate 3}, we have
	\begin{equation*}
	\begin{aligned} 
|z_i|^2\,|a_i\lrcorner(\bar\nabla^2 F_i)(z_i)| =O(1)\,\sup_{x\in S_i}|x|\,|\bar\nabla F_i|+O(|z_i|^{5/2})\,	\bigg(\int_{S_i\cap B_{|z_i|/2}(z_i)}(\Delta\nabla_{X_i}F_i)^4\,\mathrm{d}\bar\mu\bigg)^{1/4}.
	\end{aligned}
	\end{equation*}
	Note that
	$$
		\Delta \nabla_{X_i} F_i=\nabla _{X_i}\Delta F_i+K^{\Sigma}\,g({X_i},\nabla F_i)+2\,g(\nabla {X_i},\nabla^2 F_i)+g(\operatorname{tr}\nabla^2{X_i},\nabla F_i)
		$$
		where $K^\Sigma$ is the Gauss curvature of $\Sigma$. 
	Using this, Lemma \ref{support function lemma}, \eqref{asymptotic to Schwarzschild 1.5}, \eqref{kappa est}, and Corollary \ref{higher order estimates},  we obtain
	$$
\Delta \nabla_{X_i} F=	O((\lambda_i^{-1/2}+\rho_i^{-1})^2\,|x|^{-3})\,|F_i|+O(|x|^{-3})\,(|x|\,|\bar\nabla F_i|+|x|^2\,|\bar\nabla^2 F_i|).
	$$
	Consequently,
	\begin{equation*}
	\begin{aligned}
	|z_i|^{5/2}\,	\bigg(\int_{S_i\cap B_{|z_i|/2}(z_i)}(\Delta\nabla_{X_i}F_i)^4\,\mathrm{d}\bar\mu\bigg)^{1/4}
	&=O((\lambda_i^{-1/2}+\rho_i^{-1})^2)\,\sup_{x\in S_i}|F_i|+O(1)\,\sup_{x\in S_i}|x|\,|\bar\nabla F_i|
	\\& \qquad +O(|z_i|^{3/2})\,\bigg(\int_{S_i\cap B_{|z_i|/2}(z_i)}|\bar\nabla ^2 F_i|^4\,\mathrm{d}\bar\mu\bigg)^{1/4}.
	\end{aligned} 
	\end{equation*}  
Finally, by \eqref{- one} and Proposition \ref{curv est prop},
$$
\Delta F_i=O(((\lambda_i^{-1/2}+\rho_i^{-1})^2+\lambda_i^2\,|\kappa(\Sigma_i)|)\,|x|^{-2})\,|F_i|+O(|x|^{-3})\,|x|\,|\bar\nabla F_i|.
$$
By interior $L^4$-estimates, we conclude that
\begin{equation*}
\begin{aligned}
&|z_i|^{3/2}\,\bigg(\int_{S_i\cap B_{|z_i|/2}(z_i)}|\bar\nabla ^2 F_i|^4\,\mathrm{d}\bar\mu\bigg)^{1/4}
\\&\qquad=O(1)\,\sup_{x\in S_i}|x|\,|\bar\nabla F_i|+O((\lambda_i^{-1/2}+\rho_i^{-1})^2+\lambda_i^2\,|\kappa(\Sigma_i)|)\,\sup_{x\in S_i}|F_i|.
\end{aligned}
\end{equation*}
\end{rema}
\section{Variations of the Willmore energy by translations}
We assume that $g$ is a Riemannian metric on $\mathbb{R}^3$ such that, as $x\to\infty$,
\begin{align} \label{asymptotic to schwarzschild 2}
g=\left(1+|x|^{-1}\right)^4\,\bar g+\sigma\qquad\text{where}\qquad\partial_J\sigma=O(|x|^{-2-|J|})
\end{align} 
for every multi-index $J$ with $|J|\leq 4$. Let $\{\Sigma_i\}_{i=1}^\infty$ be a sequence of area-constrained Willmore spheres $\Sigma_i\subset \mathbb{R}^3$ with 
\begin{align} \label{slow divergence 2}  
m_H(\Sigma_i)\geq0, \qquad \quad
\lim_{i\to\infty}\rho(\Sigma_i)=\infty,\qquad\quad 
\rho(\Sigma_i)=o(\lambda(\Sigma_i))
\end{align} 
and assume that, as $i\to\infty$,
\begin{align} \label{s 2} 
\log\lambda(\Sigma_i)=o(\rho(\Sigma_i)).
\end{align} 
\indent As before, we abbreviate $\lambda_i=\lambda(\Sigma_i)$ and $\rho_i=\rho(\Sigma_i)$. Recall from Lemma \ref{u initial estimate} that, for all $i$ large, $\Sigma_i=\Sigma_{\xi_i,\lambda_i}(u_i)$ is the Euclidean graph of a function $u_i$ over the sphere $S_i=S_{\lambda_i}(\lambda_i\,\xi_i)$ where $\xi_i\in\mathbb{R}^3$ satisfies
\begin{align} \label{sphere rho vs lambda 2} 
|1-|\xi_i||=\lambda_i^{-1}\,\rho_i.
\end{align} 
 \indent
In this section, we compute an expansion of the variation of the Willmore energy of $\Sigma_i$ with respect to a translation in direction $\xi_i$.  \\ \indent  
 For the statement of the next lemma, we define the form
$$
\zeta_i:\Gamma(T\Sigma_i)\times \Gamma(T\Sigma_i)\to C^\infty(\mathbb{R})\qquad\text{given by}\qquad \zeta_i(X,Y)=g(D_X\xi_i,Y).
$$

\begin{lem} \label{radial 1}
There holds
\begin{align*}
&\int_{\Sigma_i} g(\xi_i,\nu)\,\big[\Delta H+|\hcirc|^2\,H+\operatorname{Ric}(\nu,\nu)\,H\big]\mathrm{d}\mu\
\\&\qquad= \int_{\Sigma_i} [ g(\operatorname{tr}_{\Sigma_i}D^2 \xi_i,\nu)+\operatorname{Ric}(\xi_i,\nu)\big]\,H\,\mathrm{d}\mu
\\&\qquad\qquad+\frac12\, \int_{\Sigma_i} \left[\operatorname{div}_{\Sigma_i} \xi_i -2\,g(D_\nu \xi_i,\nu)\right]H^2\,\mathrm{d}\mu
+2\, \int_{\Sigma_i} g(\zeta_i, \hcirc)\,H\,\mathrm{d}\mu.
\end{align*}
\end{lem}
\begin{proof}
	Note that 
	\begin{align*} 
	\Delta (g(\xi_i,\nu))=g(\operatorname{tr}_{\Sigma_i}D^2\xi_i,\nu)-g(D_\nu\xi_i,\nu)\,H+2\,g(\zeta_i,h) +(\operatorname{div}_{\Sigma_i}h)(\xi_i^\top)-g(\xi_i,\nu)\,|h|^2.
	\end{align*} 
Integrating by parts
and using the trace of the Gauss-Codazzi equation,
$$
\operatorname{div}_{\Sigma_i}h=\nabla H+\nu\lrcorner  \operatorname{Rc},
$$
we obtain
\begin{align*} 
\int_{\Sigma_i}g(\xi_i,\nu)\, \Delta H\,\mathrm{d}\mu&=\int_{\Sigma_i}\big[ g(\operatorname{tr}_{\Sigma_i}D^2 \xi_i,\nu)\,H-g(D_\nu \xi_i,\nu)\,H^2+2\,g(\zeta_i, h)\,H
\\&\qquad\qquad+H\,g\,(\xi_i,\nabla H) + \operatorname{Ric}(\xi_i^\top,\nu)\,H-g(\xi_i,\nu)\,|h|^2\,H\big]\,\mathrm{d}\mu.
\end{align*} 
Note that
\begin{align*} 
\int_{\Sigma_i} H\,g(\xi_i,\nabla H)\,\mathrm{d}\mu=\frac12\,\int_{\Sigma_i} g(\xi_i,\nabla H^2)\,\mathrm{d}\mu=\frac12\,\int_{\Sigma_i}[g(\xi_i,\nu)\,H^3-(\operatorname{div}_{\Sigma_i}\xi_i)\,H^2]\,\mathrm{d}\mu
\end{align*} 
where we have integrated by parts in the second equality. Using the decomposition
$$
h=\frac12\,H\,g|_{\Sigma_i}+\hcirc,
$$ 
we see that
$$
2\, g(\zeta_i,h)=(\operatorname{div}_{\Sigma_i}\xi_i)\,H+2\, g(\zeta_i,\hcirc)
$$
and
$$
2\,g(\xi_i,\nu)\,|h|^2=g(\xi_i,\nu)\,H^2+2\,g(\xi_i,\nu)\,|\hcirc|^2.
$$
Using that $\xi_i=\xi_i^\top+\xi_i^\perp$, we obtain 
$$
\operatorname{Ric}(\xi_i^\top,\nu)=\operatorname{Ric}(\xi_i,\nu)-g(\xi_i,\nu)\,\operatorname{Ric}(\nu,\nu).
$$\indent 
The assertion follows from these identities.

\end{proof}
For the proof of Lemma \ref{radial 4} below, let $e_1,\,e_2,\,e_3$ be the standard basis of $\mathbb{R}^3$.
\begin{lem}
	As $i\to\infty$, there holds \label{radial 4}
	\begin{equation*}
	\begin{aligned}
	\int_{\Sigma_i} \big[ g(\operatorname{tr}_{\Sigma_i}D^2 \xi_i,\nu)+\operatorname{Ric}(\xi_i,\nu)\big]\,H\,\mathrm{d}\mu
	=-8\,\pi\,\lambda_i^{-1}\,\rho_i^{-2}+	\lambda_i^{-1}\,\int_{S_{_i}} \bar g(\xi_i,\bar\nu)\,R\,\mathrm{d}\bar \mu+o(\lambda_i^{-1}\,\rho_i^{-2}).
	\end{aligned}
	\end{equation*}
\end{lem}
\begin{proof}
	Using  Lemma \ref{ambient expansions},  Lemma \ref{second a expansion}, and Lemma \ref{basic expansions}, we have
		\begin{equation*} 
	\begin{aligned}
	&\big[ \tilde g(\tilde{\operatorname{tr}}_{\Sigma_i}\tilde D^2 \xi_i,\tilde \nu)+\tilde{\operatorname{Ric}}(\xi_i,\tilde \nu)\big]\,\mathrm{d}\tilde\mu
	\\&\qquad=\bigg[4\,|x|^{-3}\,\bar g(\xi_i,\bar\nu)-12\,|x|^{-5}\,\bar g(x,\xi_i)\,\bar g(x,\bar\nu)
\\&\qquad\qquad	+4\,|x|^{-4}\,\bar g(\xi_i,\bar\nu)-8\,|x|^{-6}\,\bar g(x,\xi_i)\,\bar g(x,\bar\nu)+O(|x|^{-5}) \bigg]\mathrm{d}\bar\mu.
	\end{aligned}
	\end{equation*}
	Using also Lemma \ref{covariant derivative}, we conclude that
	\begin{equation*} 
	\begin{aligned}
	&\big[ g(\operatorname{tr}_{\Sigma_i}D^2 \xi_i,\nu)+\operatorname{Ric}(\xi_i,\nu)\big]\,\mathrm{d}\mu
	\\&\qquad=\bigg[4\,|x|^{-3}\,\bar g(\xi_i,\bar\nu)-12\,|x|^{-5}\,\bar g(x,\xi_i)\,\bar g(x,\bar\nu)
	+4\,|x|^{-4}\,\bar g(\xi_i,\bar\nu)-8\,|x|^{-6}\,\bar g(x,\xi_i)\,\bar g(x,\bar\nu) 
	\\&\qquad\qquad +\frac12\,\sum_{j=1}^3\big[(\bar D^2_{\xi_i,e_j}\sigma)(e_j,\bar\nu)-(\bar D^2_{\xi_i,\bar\nu}\sigma)(e_j,e_j)\big]+\frac12\,\sum_{j=1}^3 (\bar D^2_{\xi_i,e_j}\sigma)(e_j,\bar\nu)- (\bar D^2_{\xi_i,\bar\nu}\sigma)(\bar\nu,\bar\nu)\\&\qquad\qquad+O(|x|^{-5})\bigg]\mathrm{d}\bar\mu.
	\end{aligned}
	\end{equation*} 
By the divergence theorem,
	\begin{align} \label{div theorem} 
	\int_{\Sigma_i}\left[|x|^{-3}\,\bar g(\xi_i,\bar\nu)-3\,|x|^{-5}\,\bar g(x,\xi_i)\,\bar g(x,\bar\nu)\right]\mathrm{d}\bar\mu=0.
	\end{align} 
Note that this holds		independently of whether $\Sigma_i$ encloses the origin or not. 
	In conjunction with Lemma \ref{u initial estimate}, Lemma \ref{improved H estimate}, and  Lemma \ref{hy integral lemma},  we conclude that
	\begin{equation*}
	\begin{aligned}
	&\int_{\Sigma_i} \big[ g(\operatorname{tr}_{\Sigma_i}D^2 \xi_i,\nu)+\operatorname{Ric}(\xi_i,\nu)\big]\,H\,\mathrm{d}\mu 
	\\&\qquad= -8\,\lambda_i^{-1}\,\int_{S_i}\left[|x|^{-4}\,\bar g(\xi_i,\bar\nu)-4\,|x|^{-6}\,\bar g(x,\xi_i)\,\bar g(x,\bar\nu)\right]\mathrm{d}\bar\mu 
	\\&\qquad\qquad+\lambda_i^{-1}\,\int_{S_i}\sum_{j=1}^3\big[(\bar D^2_{\xi_i,e_j}\sigma)(e_j,\bar\nu)-(\bar D^2_{\xi_i,\bar\nu}\sigma)(e_j,e_j)\big]\,\mathrm{d}\bar\mu
	\\&\qquad\qquad+\lambda_i^{-1}\,\int_{S_i}\bigg[\sum_{j=1}^3 (\bar D^2_{\xi_i,e_j}\sigma)(e_j,\bar\nu)- (\bar D^2_{\xi_i,\bar\nu}\sigma)(\bar\nu,\bar\nu)\bigg]\,\mathrm{d}\bar\mu
	\\&\qquad\qquad+o(\lambda_i^{-1}\,\rho_i^{-2}).
	\end{aligned}
	\end{equation*}
	\indent Using Lemma \ref{inner product lemma} and Lemma \ref{integral lemma}, we have 
	\begin{align*} 
	&\lambda_i^{-1}\,\int_{S_i}\left[|x|^{-4}\,\bar g(\xi_i,\bar\nu)-4\,|x|^{-6}\,\bar g(x,\xi_i)\,\bar g(x,\bar\nu)\right]\mathrm{d}\bar\mu \\&\qquad =\pi\,\lambda_i^{-3}\,(1-|\xi_i|)^{-2}+o(\lambda_i^{-3}\,(1-|\xi_i|)^{-2}).
	\end{align*} 
	Using \eqref{sphere rho vs lambda 2}, we conclude that
	$$
	\lambda_i^{-1}\,\int_{S_i}\left[|x|^{-4}\,\bar g(\xi_i,\bar\nu)-4\,|x|^{-6}\,\bar g(x,\xi_i)\,\bar g(x,\bar\nu)\right]\,\mathrm{d}\bar\mu =\pi\,\lambda_i^{-1}\,\rho_i^{-2}+o(\lambda_i^{-1}\,\rho_i^{-2}).
	$$
	\indent 	On the one hand, applying the divergence theorem, commuting derivatives, and applying the divergence theorem again, we obtain
	\begin{align*} 
	\int_{S_i}  \sum_{j=1}^3\big[(\bar D^2_{\xi_i,e_j}\sigma)(e_j,\bar\nu)-(\bar D^2_{\xi_i,\bar\nu}\sigma)(e_j,e_j)\big]\,\mathrm{d}\bar\mu  =\int_{S_i} [\bar{\operatorname{div}}\,\bar{\operatorname{div}}\,\sigma-\bar\Delta\,\bar{\operatorname{tr}}\,\sigma]\,\bar g(\xi_i,\bar\nu)\,\mathrm{d}\bar\mu.
	\end{align*} 
	On the other hand, note that
	\begin{align*} 
	\bar{\operatorname{div}}_{S_i}(\nu\lrcorner\ \bar D_{\xi_i}\sigma) =\sum_{j=1}^3 (\bar D^2_{\xi_i,e_j}\sigma)(e_j,\bar\nu)- (\bar D^2_{\xi_i,\bar\nu}\sigma)(\bar\nu,\bar\nu)+\lambda_i^{-1}\,\bar D_{\xi_i}\operatorname{tr}\sigma-3\,\lambda_i^{-1} (\bar D_{\xi_i}\sigma)(\bar\nu,\bar\nu).
	\end{align*} 
	Consequently,
	$$
	\int_{S_i}\bigg[\sum_{j=1}^3 (\bar D^2_{\xi_i,e_j}\sigma)(e_j,\bar\nu)- (\bar D^2_{\xi_i,\bar\nu}\sigma)(\bar\nu,\bar\nu)\bigg]\,\mathrm{d}\bar\mu=O(\lambda_i^{-1}\,\rho_i^{-1}).
	$$
	In conjunction with Lemma \ref{ambient expansions}, Lemma \ref{hy integral lemma}, and \eqref{slow divergence 2},  we conclude that 
	\begin{align*}
	&\lambda_i^{-1}\,\int_{S_i}\bigg[\sum_{j=1}^3\big[(\bar D^2_{\xi_i,e_j}\sigma)(e_j,\bar\nu)-(\bar D^2_{\xi_i,\bar\nu}\sigma)(e_j,e_j)\big] +\sum_{j=1}^3 (\bar D^2_{\xi_i,e_j}\sigma)(e_j,\bar\nu)- (\bar D^2_{\xi_i,\bar\nu}\sigma)(\bar\nu,\bar\nu)\bigg]\,\mathrm{d}\bar\mu
	\\&\qquad=	\lambda_i^{-1}\,\int_{S_{i}} \bar g(\xi_i,\bar\nu)\,R\,\mathrm{d}\bar \mu+o(\lambda_i^{-1}\,\rho_i^{-2}).
	\end{align*}
	\indent The assertion follows from these estimates.
\end{proof}
For the proof of Lemma \ref{radial 2}, recall that  a tilde indicates that a geometric quantity is computed with respect to the Schwarzschild background metric with mass $2$.
\begin{lem}As $i\to\infty$, there holds\label{radial 2}
	$$
\int_{\Sigma_i}	g(\zeta_i, \hcirc)\,H\,\mathrm{d}\mu
=o(\lambda_i^{-1}\,\rho_i^{-2})
$$
and
$$
	\int_{\Sigma_i}[\operatorname{div}_{\Sigma_i}\xi_i-2\,g(D_\nu \xi_i,\nu)]\,H^2\,\mathrm{d}\mu
=o(\lambda_i^{-1}\,\rho_i^{-2}).
$$
\end{lem}
\begin{proof}
Using \eqref{asymptotic to schwarzschild 2}, Lemma \ref{covariant derivative}, Lemma \ref{basic expansions},   Proposition \ref{curv est prop}, Lemma \ref{improved H estimate}, and Lemma \ref{hy integral lemma}, we have
$$
\int_{\Sigma_i}	g(\zeta_i, \hcirc)\,H\,\mathrm{d}\mu
=2\,\lambda_i^{-1}\,\int_{\Sigma_i}	\tilde g(\tilde\zeta_i, \htildecirc)\,\mathrm{d}\tilde \mu+o(\lambda_i^{-1}\,\rho_i^{-2}).
$$
Similarly, using also \eqref{slow divergence 2},
	\begin{align*} 
	\int_{\Sigma_i}[\operatorname{div}_{\Sigma_i}\xi_i-2\,g(D_\nu \xi_i,\nu)]\,H^2\,\mathrm{d}\mu
=4\,\lambda_i^{-2}\,\int_{\Sigma_i}\left[\tilde {\operatorname{div}}_{\Sigma_i}\xi_i-2\,\tilde g(\tilde D_{\tilde \nu} \xi_i,\tilde \nu)\right]\mathrm{d}\tilde\mu+o(\lambda_i^{-1}\,\rho_i^{-2}).
\end{align*}
\indent By Lemma \ref{covariant derivative},
$$
\tilde \zeta_i=2\,(1+|x|^{-1})^{-1}\,|x|^{-3}\,\bar g(x,\xi_i)\,\tilde g|_{\Sigma_i}.
$$
Consequently,
$
\tilde g(\tilde \zeta_i, \htildecirc)=0.
$ Using Lemma \ref{covariant derivative} again, we have
$$
\tilde {\operatorname{div}}_{\Sigma_i}\xi_i=2\,\tilde g(\tilde D_{\tilde \nu} \xi_i,\tilde \nu)=4\,(1+|x|^{-1})^{-1}\,\bar g(x,\xi_i).
$$
\indent The assertion follows.
\end{proof}

\begin{lem}
There holds \label{radial 5}
$$
\kappa(\Sigma_i)\,\int_{\Sigma_i} g(\xi_i,\nu)\,H\,\mathrm{d}\mu=o(\lambda_i^{-1}\,\rho_i^{-2}).
$$
\end{lem}
\begin{proof}
	Using Lemma \ref{improved H estimate}, Lemma \ref{basic expansions}, and \eqref{asymptotic to schwarzschild 2}, we have
	\begin{align*} 
	\int_{\Sigma_i} g(\xi_i,\nu)\,H\,\mathrm{d}\mu&=\operatorname{proj}_{\Lambda_0(S_i)}H(\Sigma_i)\,\int_{\Sigma_i} \bar g(\xi_i,\bar \nu)\,\mathrm{d}\bar\mu+O(\lambda_i^{-1})\,\int_{\Sigma_i}|x|^{-1}\,\mathrm{d}\bar\mu.
	\end{align*} 
Note that
$$
	\int_{\Sigma_i}|x|^{-1}\,\mathrm{d}\bar\mu= O(\lambda_i^2\,\rho_i^{-1}).
$$
By the divergence theorem,
	$$
	\int_{\Sigma_i} \bar g(\xi_i,\bar\nu)\,\mathrm{d}\bar\mu=0.
	$$
	The assertion follows in conjunction with Lemma \ref{improved H estimate}.
\end{proof} 
\section{Proof of Theorem \ref{main result}}
Suppose, for a contradiction, that there exists a sequence $\{\Sigma_i\}_{i=1}^\infty$ of area-constrained Willmore spheres $\Sigma_i\subset M$ such that \eqref{slow divergence 2} and \eqref{s 2} hold.
 Assembling Lemma \ref{radial 1}, Lemma \ref{radial 4}, Lemma \ref{radial 2},  and Lemma \ref{radial 5}, we have
\begin{equation} 
\begin{aligned} 
0=&\,-\int_{\Sigma_i}g(\xi_i,\nu)\,[\Delta H+(|\hcirc|^2+\operatorname{Ric}(\nu,\nu)+\kappa(\Sigma_i))\,H]\,\mathrm{d}\mu\\=&\,8\,\pi\,\lambda_i^{-1}\,\rho_i^{-2}-	\lambda_i^{-1}\,\int_{S_i} \bar g(\xi_i,\bar\nu)\,R\,\mathrm{d}\bar \mu+o(\lambda_i^{-1}\,\rho_i^{-2}). \label{contradiction} 
\end{aligned}
\end{equation}
Using $R\geq-o(|x|^{-4})$ and Lemma \ref{hy integral lemma},  we find that
$$
-	\int_{S_{i}} \bar g(\xi_i,\bar\nu)\,R\,\mathrm{d}\bar \mu\geq -\int_{\{x\in S_i\,:\,\bar g(\xi_i,\bar\nu)\geq 0\}} \bar g(\xi_i,\bar\nu)\,R\,\mathrm{d}\bar \mu -o(\rho_i^{-2}).
$$ 
Moreover, using \eqref{asymptotic to schwarzschild 2} and \eqref{slow divergence 2},
$$
\int_{\{x\in S_i\,:\,\bar g(\xi_i,\bar\nu)\geq 0\}} \bar g(\xi_i,\bar\nu)\,R\,\mathrm{d}\bar \mu=O(\lambda_i^{-2})=o(\rho_i^{-2}).
$$
These estimates are incompatible with \eqref{contradiction}.

\begin{appendices}

\section{Integral curvature estimates} \label{int curv est appendix}
In \cite{kuwertSchaetzle2}, E.~Kuwert and R.~Schätzle have established integral  curvature estimates for Euclidean Willmore surfaces whose traceless second fundamental form is small in $L^2$. In this section, we adapt their method to establish integral curvature estimates for large area-constrained Willmore spheres   in Riemannian three-manifolds which are asymptotic to Schwarzschild whose curvature is small in $L^2$. \\ \indent  In short, we use integration by parts, the area-constrained Willmore equation, and the $L^2$-estimate \eqref{small curvature} to prove local $W^{2,2}$-bounds for the second fundamental form $h$. In conjunction with the Sobolev inequality, we obtain an $L^\infty$-estimate for $h$. Compared to \cite{kuwertSchaetzle2}, additional curvature terms owing to the non-flat background need to be addressed. \\ \indent  We assume that $g$ is a Riemannian metric on $\mathbb{R}^3$ such that, as $x\to\infty,$
$$
g=\left(1+|x|^{-1}\right)^4\,\bar g+\sigma \qquad\text{where}\qquad \partial_J\sigma=O(|x|^{-2-|J|})
$$ for every multi-index $J$ with $|J|\leq 4$. \\ \indent 
Let $\{\Sigma_i\}_{i=1}^\infty$ be a sequence of spheres $\Sigma_i\subset \mathbb{R}^3$ which satisfy the area-constrained Willmore equation \eqref{constrained Willmore equation} with
\begin{align*}
\lim_{i\to\infty}\rho(\Sigma_i)=\infty,\qquad \rho(\Sigma_i)=O(\lambda(\Sigma_i)),
\end{align*} 
and
\begin{align} \label{small curvature}
	\int_{\Sigma_i}|h-\lambda(\Sigma_i)^{-1}\,g_{\Sigma_i}|^2\,\mathrm{d}\mu=o(1).
\end{align}
Here, $\lambda(\Sigma_i)$ and $\rho(\Sigma_i)$ are the area radius and inner radius of $\Sigma_i$ defined in \eqref{inner and outer radius}.
We abbreviate $\rho_i=\rho(\Sigma_i)$ and $\lambda_i=\lambda(\Sigma_i)$. \\ \indent 
\indent 
Let $X,\,Y,\,Z$ be vector fields tangent to $\Sigma$ and $\operatorname{Rm}$ the Riemann curvature tensor of $(M,g)$.  We recall the Gauss-Codazzi equation
$$
(\nabla_{X}h)(Y,Z)=(\nabla_Yh)(X,Z)+\operatorname{Rm}(X,Y,\nu,Z)
$$
and its trace
\begin{align} \label{gauss codazzi}
\operatorname{div}_{\Sigma_i}h=\nabla H+\nu \lrcorner\ \operatorname{Rc}.
\end{align} 
For the statement of the following Simons-type identities, note that  the contraction of the divergence is with respect to the first entry. Given a covariant tensor $T$, we follow the convention that
$$
X\lrcorner\ (\nabla T)=(\nabla_X T).
$$ 
\begin{lem}[{\cite[Lemma 3.2]{koerber}}]
There holds,  on $\Sigma_i$,
	\begin{align}
\label{first simon}	\Delta\,\hcirc=\,&{\accentset{\circ}{\nabla^2 H}}+\frac12\,H^2\,\hcirc+\hcirc*\hcirc*\hcirc+O(|x|^{-3}\,|h|)+O(|x|^{-4}),\\
	\operatorname{div}_{\Sigma_i}\nabla^2H=\,&\nabla\Delta H+\frac14\,H^2\,\nabla H+\hcirc*\hcirc*\nabla H+O(|x|^{-3}\,|\nabla H|), \label{second simon} \\
\operatorname{div}_{\Sigma_i}\nabla^2\hcirc=\,&\nabla\Delta\hcirc+ h* h*\nabla \hcirc\label{third simon}\\&\qquad \notag+O(|x|^{-3}\,|\hcirc|\,|h|)+O(|x|^{-4}\,|\hcirc|)+O(|x|^{-3}\,|\nabla\hcirc|). 
	\end{align}
\end{lem}
	Let $\psi\in C^\infty(\mathbb{R})$ with
	\begin{itemize} 
		\item[$\circ$] 	 $0\leq \psi \leq 1$,
		\item[$\circ$] $\psi(1)=1$,
		\item[$\circ$] $\psi(s)=0$ if $s<3/4$ or $s>5/4$,
		\item[$\circ$] $|\psi'|\leq \frac92$.
		\end{itemize} 
	 We fix $x\in\mathbb{R}^3$ with $x\neq 0$ and define $\eta\in C^\infty(\mathbb{R}^3)$ by
\begin{equation} \label{eta def} 
\eta(z)=\psi(|z|\,|x|^{-1}).
\end{equation} 
Note that
\begin{align} \label{eta estimate}
|D \eta|\leq 5\,|x|^{-1}.
\end{align} 
Moreover, by Lemma \ref{basic expansions}, \eqref{area estimate}, and Lemma \ref{large blowdown}, we have, uniformly for all $x\in \Sigma_i$,
\begin{align} \label{eta area estimate}
|\Sigma_i\cap\operatorname{spt}(\eta)| =O(|x|^{2}).
\end{align}
\indent The following lemma is an adaptation of \cite[Lemma 2.2]{kuwertSchaetzle2}.
\begin{lem}
	There holds, 	uniformly for all $x\in\Sigma_i$, \label{lem aid 0}
	\begin{align*} 
\int_{\Sigma_i}\eta^2\,|\nabla \hcirc|^2\,\mathrm{d}\mu&	\leq \frac12\,\int_{\Sigma_i}\eta^2\,H^2\,|\hcirc|^2\,\mathrm{d}\mu+ O(1)\int_{\Sigma_i}\eta^2\,|\hcirc|^4\,\mathrm{d}\mu+O(|x|^{-2})\,\int_{\Sigma_i\cap\operatorname{spt}(\eta)}|h-\lambda_i^{-1}\,g|_{\Sigma_i}|^2\,\mathrm{d}\mu \\&\qquad+O(|x|^{-4})
+O\big(\kappa(\Sigma_i)^2\,|x|^2\big).
	\end{align*}
\end{lem}
\begin{proof}
We multiply \eqref{first simon} by $\eta^2\,\hcirc$ and integrate by parts. Using \eqref{gauss codazzi}, \eqref{eta estimate}, and \eqref{eta area estimate}, we obtain
	\begin{align*}
	&\int_{\Sigma_i}\eta^2\,\left[|\nabla \hcirc|^2+\frac12 H^2\,|\hcirc|^2\right]\mathrm{d}\mu\\&\qquad=\frac12\,\int_{\Sigma_i}\eta^2\,|\nabla H|^2\,\mathrm{d}\mu+O(|x|^{-1})\,\int_{\Sigma_i}\eta\,|\hcirc|\,|\nabla\hcirc|\,\mathrm{d}\mu+O(1)\,\int_{\Sigma_i}\eta^2\,|\hcirc|^4\,\mathrm{d}\mu\\&\qquad\qquad+O(|x|^{-4})\,\int_{\Sigma_i}\eta\,|\hcirc|\,\mathrm{d}\mu+O(|x|^{-3})\,\int_{\Sigma_i}\eta^2\,| \hcirc|\,|h|\,\mathrm{d}\mu\\&\qquad\qquad+O(|x|^{-3})\,\int_{\Sigma_i}\eta^2\,|\nabla \hcirc|\,\mathrm{d}\mu+O(|x|^{-6})\,\int_{\Sigma_i}\eta^2\,\mathrm{d}\mu.
	\end{align*}
Note that
	\begin{align*} 
&	O(|x|^{-1})\,\int_{\Sigma_i}\eta\,|\hcirc|\,|\nabla\hcirc|\,\mathrm{d}\mu+O(|x|^{-4})\,\int_{\Sigma_i}\eta\,|\hcirc|\,\mathrm{d}\mu+O(|x|^{-3})\,\int_{\Sigma_i}\eta^2\,|\nabla \hcirc|\,\mathrm{d}\mu\\&\qquad\leq\, \frac12 \int_{\Sigma_i}\eta^2\,|\nabla\hcirc|^2\,\mathrm{d}\mu+O(|x|^{-2})\,\int_{\Sigma_i\cap\operatorname{spt}(\eta)}|\hcirc|^2\,\mathrm{d}\mu+O(|x|^{-6})\,\int_{\Sigma_i}\eta^2\,\mathrm{d}\mu.
	\end{align*}
Likewise, using Lemma \ref{L2 curv est rema}, 
	\begin{align*} 
	O(|x|^{-3})\,\int_{\Sigma_i}\eta^2\,| \hcirc|\,|h|\,\mathrm{d}\mu
	\leq\,O(|x|^{-2})\,\int_{\Sigma_i\cap\operatorname{spt}(\eta)}|\hcirc|^2\,\mathrm{d}\mu+O(|x|^{-4}).
	\end{align*} 
	Integrating by parts we have
	\begin{align*} 
	\int_{\Sigma_i}\eta^2\,|\nabla H|^2\,\mathrm{d}\mu&=-\int_{\Sigma_i}\eta^2\,(H-2\,\lambda_i^{-1})\,\Delta H\,\mathrm{d}\mu -2\,\int_{\Sigma_i}\eta\,(H-2\,\lambda_i^{-1})\,g(\nabla \eta,\nabla H)\,\mathrm{d}\mu.
	\end{align*} 
	Using \eqref{eta estimate}, we estimate
	\begin{align*} 
	-2\,\int_{\Sigma_i}\eta\,(H-2\,\lambda_i^{-1})\,g(\nabla \eta,\nabla H)\,\mathrm{d}\mu\leq \frac1{2}\int_{\Sigma_i}\eta^2\,|\nabla H|^2\,\mathrm{d}\mu+O\left(|x|^{-2}\right)\,\int_{\Sigma_i\cap\operatorname{spt}(\eta)}(H-2\,\lambda_i^{-1})^2\,\mathrm{d}\mu.
	\end{align*} 
	Using \eqref{constrained Willmore equation}, we have
	\begin{align*} 
	&-\int_{\Sigma_i}\eta^2\,(H-2\,\lambda_i^{-1})\,\Delta H\,\mathrm{d}\mu\\&\qquad =\int_{\Sigma_i}\eta^2\,(H-2\,\lambda_i^{-1})\,H\,|\hcirc|^2\,\mathrm{d}\mu+O(|x|^{-3}+|\kappa(\Sigma_i)|)\,\int_{\Sigma_i}\eta^2\,|H-2\,\lambda_i^{-1}|\,|H|\,\mathrm{d}\mu.
	\end{align*}
	Note that
	\begin{align*}
\int_{\Sigma_i}\eta^2\,(H-2\,\lambda_i^{-1})\,H\,|\hcirc|^2\,\mathrm{d}\mu =\int_{\Sigma_i}\eta^2\,H^2\,|\hcirc|^2\,\mathrm{d}\mu+\int_{\Sigma_i}\eta^2\,[\lambda_i^{-1}\,(H-2\,\lambda_i^{-1})+2\,\lambda_i^{-2}]\,|\hcirc|^2\,\mathrm{d}\mu
	\end{align*}
	and 
	\begin{align*} 
	&\int_{\Sigma_i}\eta^2\,[\lambda_i^{-1}\,(H-2\,\lambda_i^{-1})+2\,\lambda_i^{-2}]\,|\hcirc|^2\,\mathrm{d}\mu\\	
	&\qquad\leq O(1)\,\int_{\Sigma_i}\eta^2\,|\hcirc|^4\,\mathrm{d}\mu+O(|x|^{-2})\,\int_{\Sigma_i\cap\operatorname{spt}(\eta)}|\hcirc|^2\,\mathrm{d}\mu+		O(|x|^{-2})\,\int_{\Sigma_i\cap\operatorname{spt}(\eta)}(H-2\,\lambda_i^{-1})^2\,\mathrm{d}\mu.
	\end{align*} 
Moreover, using \eqref{small curvature},
\begin{align*} 
&O(|x|^{-3}+|\kappa(\Sigma_i)|)\,\int_{\Sigma_i}\eta^2\,|H-2\,\lambda_i^{-1}|\,|H|\,\mathrm{d}\mu \\&\qquad \leq O(|x|^{-4})+O(\kappa(\Sigma_i)^2\,|x|^2)+	O(|x|^{-2})\,\int_{\Sigma_i\cap\operatorname{spt}(\eta)}(H-2\,\lambda_i^{-1})^2\,\mathrm{d}\mu.
\end{align*}
The assertion follows from these estimates and \eqref{area estimate}.
\end{proof}  
\begin{coro}
	There holds, uniformly for all $x\in\Sigma_i$, \label{coro aid}
	\begin{align*} 
	\int_{\Sigma_i}\eta^2\,|\nabla H|^2\,\mathrm{d}\mu &		\leq4\,	\int_{\Sigma_i}\eta^2\,H^2\,|\hcirc|^2\,\mathrm{d}\mu+ O(1)\,\int_{\Sigma_i}\eta^2\,|\hcirc|^4\,\mathrm{d}\mu+O(|x|^{-2})\,\int_{\Sigma_i\cap\operatorname{spt}(\eta)}|h-\lambda_i^{-1}\,g|_{\Sigma_i}|^2\,\mathrm{d}\mu \\&\qquad\qquad+O(|x|^{-4})+O(\kappa(\Sigma_i)^2\,|x|^2).
	\end{align*}
\end{coro}
\begin{proof}
	This follows from Lemma \ref{lem aid 0} and \eqref{gauss codazzi}.
\end{proof}
The next two lemmas follow \cite[Lemma 2.3]{kuwertSchaetzle2}.
\begin{lem}
There holds, uniformly for all $x\in\Sigma_i$, \label{lem aid}
\begin{align*} 
&\int_{\Sigma_i}\eta^4\,H^2\,|\nabla\hcirc|^2\,\mathrm{d}\mu+\int_{\Sigma_i}\eta^4\,H^4\,|\hcirc|^2\,\mathrm{d}\mu\\&\qquad\leq\,2\,\int_{\Sigma_i}\eta^4\,H^2\,|\nabla H|^2\,\mathrm{d}\mu+O(1)\,\int_{\Sigma_i}\eta^4\, H^2\,|\hcirc|^4\,\mathrm{d}\mu\\&\qquad\qquad+ O(1)\,\int_{\Sigma_i}\eta^4\, |\hcirc|^2\,|\nabla\hcirc|^2\,\mathrm{d}\mu +O(|x|^{-4})\,\int_{\Sigma_i\cap\operatorname{spt}(\eta)}|\hcirc|^2\,\mathrm{d}\mu +O(|x|^{-6}).
\end{align*}

\end{lem}
\begin{proof}
We multiply \eqref{first simon} by $\eta^4\,H^2\,\hcirc$ and integrate by parts. Using \eqref{gauss codazzi} and \eqref{eta estimate}, we have	
\begin{align*} 
&\int_{\Sigma_i}\eta^4\,H^2\,|\nabla\hcirc|^2\,\mathrm{d}\mu+\frac12\int_{\Sigma}\eta^4\,H^4\,|\hcirc|^2\,\mathrm{d}\mu\\&\qquad\leq\frac12\int_{\Sigma}\eta^4\,H^2\,|\nabla H|^2\,\mathrm{d}\mu+O(1)\,\int_{\Sigma_i}\eta^4\, H^2\,|\hcirc|^4\,\mathrm{d}\mu
\\&\qquad\qquad +O(|x|^{-1})\,\int_{\Sigma_i}\eta^3\,H^2\,|\hcirc|\,|\nabla\hcirc|\,\mathrm{d}\mu+O(|x|^{-1})\,\int_{\Sigma_i}\eta^3\,H^2\,|\hcirc|\,|\nabla H|\,\mathrm{d}\mu
\\&\qquad\qquad+
 O(1)\,\int_{\Sigma_i}\eta^4\, |H|\,|\hcirc|\,|\nabla H|\,|\nabla\hcirc|\,\mathrm{d}\mu+O(1)\,\int_{\Sigma_i}\eta^4\, |H|\,|\hcirc|\,|\nabla H|^2\,\mathrm{d}\mu
 \\&\qquad\qquad+O(|x|^{-3})\,\int_{\Sigma_i}\eta^4\, H^2\,|\nabla H|\,\mathrm{d}\mu+O(|x|^{-3})\,\int_{\Sigma_i}\eta^4\, H^2\,|\hcirc|\,|h|\,\mathrm{d}\mu\\&\qquad\qquad+O(|x|^{-4})\,\int_{\Sigma_i}\eta^4\, H^2\,|\hcirc|\,\mathrm{d}\mu.
\end{align*} 
Note that
\begin{align*} 
&O(|x|^{-1})\,\int_{\Sigma_i}\eta^3\,H^2\,|\hcirc|\,|\nabla\hcirc|\,\mathrm{d}\mu+O(|x|^{-1})\,\int_{\Sigma_i}\eta^3\,H^2\,|\hcirc|\,|\nabla H|\,\mathrm{d}\mu
\\&\qquad\leq\frac1{16}\,\int_{\Sigma_i}\eta^4\,H^2\,|\nabla H|^2\,\mathrm{d}\mu +\frac1{16}\,\int_{\Sigma_i}\eta^4\,H^2\,|\nabla \hcirc|^2\,\mathrm{d}\mu\\&\qquad\qquad+\frac1{16}\,\int_{\Sigma_i}\eta^4\,H^4\,|\hcirc|^2\,\mathrm{d}\mu
+O(|x|^{-4})\,\int_{\Sigma_i\cap\operatorname{spt}(\eta)} |\hcirc|^2\,\mathrm{d}\mu.
\end{align*} 
Likewise,
\begin{align*}
O(1)\,\int_{\Sigma_i}\eta^4\, H\,|\hcirc|\,|\nabla H|\,|\nabla\hcirc|\,\mathrm{d}\mu\leq \frac1{32}\,\int_{\Sigma_i}\eta^4\,H^2\,|\nabla H|^2\,\mathrm{d}\mu + O(1)\,\int_{\Sigma_i}\eta^4\, |\hcirc|^2\,|\nabla\hcirc|^2\,\mathrm{d}\mu.
\end{align*} 
Moreover, using \eqref{gauss codazzi},
\begin{align*}
&O(1)\,\int_{\Sigma_i}\eta^4\, |H|\,|\hcirc|\,|\nabla H|^2\,\mathrm{d}\mu
\\&\qquad\leq\frac1{32}\,\int_{\Sigma_i}\eta^4\,H^2\,|\nabla H|^2\,\mathrm{d}\mu  + O(1)\,\int_{\Sigma_i}\eta^4\, |\hcirc|^2\,|\nabla\hcirc|^2\,\mathrm{d}\mu+O(|x|^{-6})\,\int_{\Sigma_i\cap\operatorname{spt}(\eta)} |\hcirc|^2\,\mathrm{d}\mu.
\end{align*}
Finally, using  \eqref{small curvature}, 
\begin{align*}
&O(|x|^{-3})\,\int_{\Sigma_i}\eta^4\,H^2\,|\nabla H|\,\mathrm{d}\mu +O(|x|^{-3})\,\int_{\Sigma_i}\eta^4\, H^2\,|\hcirc|\,|h|\,\mathrm{d}\mu+O(|x|^{-4})\,\int_{\Sigma_i}\eta^4\, H^2\,|\hcirc|\,\mathrm{d}\mu\\ &\qquad \qquad \leq \frac1{16}\,\int_{\Sigma_i}\eta^4\,H^2\,|\nabla H|^2\,\mathrm{d}\mu 
+\frac1{16}\,\int_{\Sigma_i}\eta^4\,H^4\,|\hcirc|^2\,\mathrm{d}\mu +O(|x|^{-6})+O(|x|^{-8})\,\int_{\Sigma_i}\eta^4\,\mathrm{d}\mu.
\end{align*}
The assertion follows from these estimates and \eqref{area estimate}.
\end{proof}

\begin{lem}
	There holds, 	 uniformly for all $x\in\Sigma_i$, \label{nabla2 h int estimate}
	\begin{align*}
	&\int_{\Sigma_i}\eta^4\,[|\nabla^2H|^2+|h|^2\,|\nabla h|^2+|h|^4\,|\hcirc|^2]\,\mathrm{d}\mu \\&\qquad\leq O(1)\,\int_{\Sigma_i}\eta^4\,|\hcirc|^2\,|\nabla\hcirc|^2\,\mathrm{d}\mu+O(1)\,\int_{\Sigma_i}\eta^4\,|\hcirc|^6\,\mathrm{d}\mu
	 +O(|x|^{-4})\,\int_{\Sigma_i\cap\operatorname{spt}(\eta)}|h-\lambda_i^{-1}\,g|_{\Sigma_i}|^2\,\mathrm{d}\mu
	\\&\qquad\qquad+O(|x|^{-6})+O(\kappa(\Sigma_i)^2).
	\end{align*}
\end{lem}
\begin{proof}
	We multiply \eqref{second simon} by $\eta^4\,\nabla H$ and integrate by parts. Using \eqref{gauss codazzi} and \eqref{eta estimate}, we obtain
\begin{align*}
&\int_{\Sigma_i} \eta^4\,\left[|\nabla^2H|^2+\frac14\,H^2\,|\nabla H|^2\right]\mathrm{d}\mu\\&\qquad \leq\int_{\Sigma_i}\eta^4\,(\Delta H)^2\,\mathrm{d}\mu+40\,|x|^{-1}\,\int_{\Sigma_i}\eta^3\,|\nabla H|\,|\nabla^2H|\,\mathrm{d}\mu
 +O(1)\,\int_{\Sigma_i}\eta^4\,|\hcirc|^2\,|\nabla\hcirc|^2\,\mathrm{d}\mu\\&\qquad\qquad+O(|x|^{-6})\,\int_{\Sigma_i}\eta^4\,|\hcirc|^2\,\mathrm{d}\mu +O(|x|^{-3})\,\int_{\Sigma_i}\eta^4\,|\nabla H|^2\,\mathrm{d}\mu.
\end{align*}
Using \eqref{constrained Willmore equation} and \eqref{small curvature}, we conclude that 
\begin{align*}
&\int_{\Sigma_i} \eta^4\,\left[|\nabla^2H|^2+\frac14\,H^2\,|\nabla H|^2\right]\mathrm{d}\mu\\&\qquad \leq \int_{\Sigma}\eta^4\,H^2\,|\hcirc|^4\,\mathrm{d}\mu+\frac12\,\int_{\Sigma_i}\eta^4\,|\nabla^2 H|^2\,\mathrm{d}\mu+10^3\,|x|^{-2}\,\int_{\Sigma_i}\eta^4\,|\nabla H|^2\,\mathrm{d}\mu \\&\qquad\qquad+O(|x|^{-6})+O(\kappa(\Sigma_i)^2)+O(1)\,\int_{\Sigma_i}\eta^4\,|\hcirc|^2\,|\nabla \hcirc|^2\,\mathrm{d}\mu +O(|x|^{-6})\,\int_{\Sigma_i}\eta^4\,|\hcirc|^2\,\mathrm{d}\mu.
\end{align*}
Note that
\begin{align*}
\int_{\Sigma}\eta^4\,H^2\,|\hcirc|^4\,\mathrm{d}\mu\leq \frac1{32}\int_{\Sigma}\eta^4\,H^4\,|\hcirc|^2\,\mathrm{d}\mu+O(1)\,\int_{\Sigma}\eta^4\,|\hcirc|^6\,\mathrm{d}\mu
\end{align*}
and 
\begin{align*} 
	\int_{\Sigma_i}\eta^2\,H^2\,|\hcirc|^2\,\mathrm{d}\mu\leq \frac{1}{128}\,10^{-3}\,|x|^2\,\int_{\Sigma_i}\eta^4\,H^4\,|\hcirc|^2\,\mathrm{d}\mu +O(|x|^{-2})\,\int_{\Sigma_i\cap\operatorname{spt}(\eta)}|\hcirc|^2\,\mathrm{d}\mu.
\end{align*}
The assertion follows from these estimates, Corollary \ref{coro aid},  Lemma \ref{lem aid}, and \eqref{gauss codazzi}.
\end{proof} 
The following lemma is an adaptation of \cite[Proposition 2.4]{kuwertSchaetzle2}.
\begin{lem}
	There holds, uniformly for all $x\in\Sigma_i$, \label{final integral estimate} 
	\begin{align*}
	&\int_{\Sigma_i}\eta^4\,[|\nabla^2h|^2+|h|^2\,|\nabla h|^2+|h|^4\,|\hcirc|^2]\,\mathrm{d}\mu \\&\qquad \leq
	O(|x|^{-4})\,\int_{\Sigma_i\cap\operatorname{spt}(\eta)}|h-\lambda_i^{-1}\,g|_{\Sigma_i}|^2\,\mathrm{d}\mu
	+O(|x|^{-6})+O(\kappa(\Sigma_i)^2).
	\end{align*}
\end{lem}
\begin{proof}
	We multiply \eqref{third simon} by $\eta^4\,\nabla\hcirc$ and integrate by parts. Using Lemma \ref{L2 curv est rema} and \eqref{eta estimate}, we find that
	\begin{align*} 
	\int_{\Sigma_i}\eta^4|\nabla ^2\hcirc|^2\,\mathrm{d}\mu&\leq \int_{\Sigma_i}\eta^4\,|\Delta\hcirc|^2\,\mathrm{d}\mu+\frac14\,\int_{\Sigma_i}\eta^4\,|\nabla^2\hcirc|^2\,\mathrm{d}\mu\\&\qquad +O(|x|^{-2})\,\int_{\Sigma_i}|\nabla \hcirc|^2\,\mathrm{d}\mu+O(1)\,\int_{\Sigma_i}\eta^4\,|h|^2\,|\nabla h|^2\,\mathrm{d}\mu
	+O(|x|^{-6}).
	\end{align*} 
	Using \eqref{first simon} and Lemma \ref{L2 curv est rema}, we obtain
	\begin{align*}
	\int_{\Sigma_i}\eta^4\,|\Delta\hcirc|^2\,\mathrm{d}\mu&\leq\ \int_{\Sigma_i}\eta^4\,|\nabla^2H|^2\,\mathrm{d}\mu+O(1)\,\int_{\Sigma_i}\eta^4\,H^4\,|\hcirc|^2\,\mathrm{d}\mu\\&\qquad +O(1)\,\int_{\Sigma_i}\eta^4\,|\hcirc|^6\,\mathrm{d}\mu+O(|x|^{-6})+ O(|x|^{-8})\,\int_{\Sigma_i}\eta^2\,\mathrm{d}\mu.
	\end{align*}
Assembling these estimates and using Lemma \ref{nabla2 h int estimate}, Lemma \ref{lem aid 0}, and \eqref{eta area estimate}, we have
	\begin{align*}
&\int_{\Sigma_i}\eta^4\,[|\nabla^2h|^2+|h|^2\,|\nabla h|^2+|h|^4\,|\hcirc|^2]\,\mathrm{d}\mu \\&\qquad\leq O(1)\,\int_{\Sigma_i}\eta^4\,|\hcirc|^2\,|\nabla\hcirc|^2\,\mathrm{d}\mu+O(1)\,\int_{\Sigma_i}\eta^4\,|\hcirc|^6\,\mathrm{d}\mu
\\&\qquad\qquad +O(|x|^{-4})\,\int_{\Sigma_i\cap\operatorname{spt}(\eta)}|h-\lambda_i^{-1}\,g|_{\Sigma_i}|^2\,\mathrm{d}\mu
 +O(|x|^{-6})+O(\kappa(\Sigma_i)^2).
\end{align*}
The argument now concludes as in \cite[Lemma 2.5 and Proposition 2.6]{kuwertSchaetzle2}, using \eqref{small curvature} and the Michael-Simon Sobolev inequality in the form \cite[Proposition 5.4]{HuiskenYau}.
\end{proof}
\begin{prop}
	There holds, uniformly for all $x\in\Sigma_i$, \label{int curv estimate}
		\begin{align*} 
	|h-\lambda_i^{-1}\,g|_{\Sigma_i}|^4=\,&O(|x|^{-4})\,\bigg(\int_{\Sigma_i\cap B_{|x|/4}(x)}|h-\lambda_i^{-1}\,g|_{\Sigma_i}|^2\,\mathrm{d}\mu\bigg)^2	
	\\&+O(|x|^{-8})+O(\kappa(\Sigma_i)^2)\,\int_{\Sigma_i\cap B_{|x|/4}(x)}|h-\lambda_i^{-1}\,g|_{\Sigma_i}|^2\,\mathrm{d}\mu.
	\end{align*} 

\end{prop}
\begin{proof}
	Repeating the argument that led to \cite[Lemma 2.8]{kuwertSchaetzle2} using \cite[Proposition 5.4]{HuiskenYau}, we find that
	\begin{align*} 
	|\eta^2\,\hcirc|_{L^\infty(\Sigma_i)}^4\leq O(1)\,\int_{\Sigma_i\cap\operatorname{spt}(\eta)}|\hcirc|^2\,\mathrm{d}\mu\,\left[\int_{\Sigma_i}\eta^8\,[|\nabla^2\hcirc|^2+H^4\,|\hcirc|^2]\,\mathrm{d}\mu+|x|^{-4}\,\int_{\Sigma_i\cap\operatorname{spt}(\eta)}|\hcirc|^2\,\mathrm{d}\mu\right]
	\end{align*} 
	and
	\begin{align*} 
	|\eta^2\,(H-2\,\lambda_i^{-1})|_{L^\infty(\Sigma_i)}^4&\leq O(1)\,\int_{\Sigma_i\cap\operatorname{spt}(\eta)}|\hcirc|^2\,\mathrm{d}\mu\,\bigg[\int_{\Sigma_i}\eta^8\,[|\nabla^2H|^2+H^4\,(H-2\,\lambda_i^{-1})^2]\,\mathrm{d}\mu\\&\qquad\qquad\qquad\qquad\,\,+|x|^{-4}\,\int_{\Sigma_i\cap\operatorname{spt}(\eta)}(H-2\,\lambda_i^{-1})^2\,\mathrm{d}\mu\bigg].
	\end{align*} 
	By \eqref{small curvature},
	\begin{align*} 
	\int_{\Sigma_i}\eta^8\,H^4\,(H-2\,\lambda_i^{-1})^2\,\mathrm{d}\mu\leq \,o(1)\,|\eta^2(H-2\,\lambda_i^{-1})|_{L^\infty(\Sigma_i)}^4+O(\lambda_i^{-4})\,\int_{\Sigma_i\cap\operatorname{spt}(\eta)}(H-2\,\lambda_i^{-1})^2\,\mathrm{d}\mu.
	\end{align*} 
By \eqref{eta def},
	$$
	\operatorname{spt}(\eta)\subset \cap B_{|x|/4}(x).
	$$
	The assertion follows from these estimates and Lemma \ref{final integral estimate}.
	\end{proof}
\section{Surfaces with bounded Euclidean Willmore energy}
In this section, we recall  estimates  for closed surfaces in $\mathbb{R}^3$ i	n terms of a bound on their Euclidean Willmore energy.
\begin{lem}[{\cite[Lemma 5.2]{HuiskenYau}}] For every $q>2$ there is a constant $c(q)>0$ such that for every closed surface  $\Sigma\subset\mathbb{R}^3\setminus\{0\}$, \label{hy integral lemma}
	\begin{align*}
		\rho(\Sigma)^{q-2}\,\int_{\Sigma}|x|^{-q}\,\mathrm{d}\bar\mu\leq c(q)\,\int_{\Sigma}\bar H^2\,\mathrm{d}\bar\mu.
	\end{align*}
\end{lem} 
The estimates in the following lemma are stated in \cite{Simon} except for the explicit constants. We revisit the proof in \cite{Simon} and compute explicit constants below. 
\begin{lem}[{\cite[Lemma 1.1 and (1.3)]{Simon}}] 
	Let $\Sigma\subset\mathbb{R}^3$ be a closed surface. Given $x\in \Sigma$ and $r>0$, we have
	\begin{align} \label{area estimate} 
		r^{-2}\,|\Sigma\cap B_r(x)|\leq\frac{3+2\,\sqrt{2}}{16}\, \int_{\Sigma}\bar H^2\,\mathrm{d}\bar\mu.
	\end{align}
	Moreover
	\begin{align} \label{diam estimate}
		\sup\{|y-z|^2:y,z\in\Sigma\}\leq \frac{17^2\,3^4}{2^6\,\pi^2}\, |\Sigma|_{\bar g}\,\int_{\Sigma}\bar H^2\,\mathrm{d}\bar\mu.
	\end{align} 
\end{lem}
\begin{proof} Let $x\in \Sigma$.
	Recall from \cite[(1.2)]{Simon} that, for all $0<r\leq t$, 
	\begin{equation} \label{simon equation}
		\begin{aligned} 
			&	r^{-2}\,|\Sigma\cap B_{r}(x)|_{\bar g}\\&\qquad\leq t^{-2}\,|\Sigma\cap B_{t}(x)|_{\bar g}+\frac1{16}\,\int_{\Sigma\cap B_{t}(x)}\bar H^2\,\mathrm{d}\bar\mu\\&\qquad\qquad+\frac12\,t^{-2}\int_{\Sigma\cap B_{t}(x)}\bar H\,\bar g(z-x,\bar\nu)\,\mathrm{d}\bar\mu(z)-\frac12\,r^{-2}\int_{\Sigma\cap B_{r}(x)}\bar H\,\bar g(z-x,\bar\nu)\,\mathrm{d}\bar\mu(z).
		\end{aligned} 
	\end{equation}
	\indent 
	To show \eqref{diam estimate}, we obtain, using the estimates \begin{align*} 
		\bigg|\int_{\Sigma\cap B_{r}(x)}\bar H\,\bar g(z-x,\bar\nu)\,\mathrm{d}\bar\mu(z)\bigg|\leq \frac38\,r^2\,\int_{\Sigma\cap B_t(x)}\bar H^2\,\mathrm{d}\bar\mu+\frac23\,|\Sigma\cap B_{r}(x)|
	\end{align*} 
	and 
	\begin{align*}  
		\bigg|\int_{\Sigma\cap B_{t}(x)}\bar H\,\bar g(z-x,\bar\nu)\,\mathrm{d}\bar\mu(z)\bigg|\leq \frac14\,t^2\,\int_{\Sigma\cap B_t(x)}\bar H^2\,\mathrm{d}\bar\mu+|\Sigma\cap B_{t}(x)|,
	\end{align*} 
	that 
	\begin{align} \label{instead of 1.3}
		r^{-2}\,|\Sigma\cap B_r(x)|_{\bar g}\leq \frac94\,\left( t^{-2}\,|\Sigma\cap B_t(x)|_{\bar g}+\frac14\,\int_{\Sigma\cap B_t(x)}\bar H^2\,\mathrm{d}\bar\mu\right) .
	\end{align} 
	Revisiting the proof of \cite[Lemma 1.1]{Simon} and using  \cite[(1.3)]{Simon} with the explicit constant $C=9/4$ computed in \eqref{instead of 1.3}, we obtain \eqref{diam estimate}. \\ \indent 
	To obtain \eqref{area estimate}, we  let $t\to\infty$ in \eqref{simon equation} and estimate
	\begin{align*} 
		r^{-2}\,|\Sigma\cap B_{r}(x)|_{\bar g}&\leq \frac1{16}\,\int_{\Sigma}\bar H^2\,\mathrm{d}\bar\mu-\frac12\,r^{-2}\int_{\Sigma\cap B_{r}(x)}\bar H\,\bar g(z-x,\bar\nu)\,\mathrm{d}\bar\mu(z)
		\\&\leq \frac{1}{16}\,\frac{\sqrt{2}}{\sqrt{2}-1}\,\int_{\Sigma}\bar H^2+(\sqrt{2}-1)\,r^{-2}\,|\Sigma\cap B_r(x)|_{\bar g}.
	\end{align*} 
\end{proof}

\section{Geometric identities on round spheres}

\begin{lem} \label{integral lemma}
Let $\xi\in\mathbb{R}^3$. The following hold if $|\xi|<1$.
\begin{align*} 
&\circ\qquad\int_{S_1(\xi)}|x|^{-1}\,\mathrm{d}\bar\mu=\,4\,\pi \\ &\circ\qquad\int_{S_1(\xi)}|x|^{-3}\,\mathrm{d}\bar\mu=\,4\,\pi\,(1-|\xi|^2)^{-1}\qquad\qquad\qquad\quad\,\,\\ &\circ\qquad\int_{S_1(\xi)}|x|^{-5}\,\mathrm{d}\bar\mu=\,\frac{4\,\pi}{3}\,(3+|\xi|^2)\,(1-|\xi|^2)^{-3}
\end{align*} 
The following hold if $|\xi|>1$.
\begin{align*} 
&\circ\qquad\int_{S_1(\xi)}|x|^{-1}\,\mathrm{d}\bar\mu=\,4\,\pi\,|\xi|^{-1} \\ &\circ\qquad\int_{S_1(\xi)}|x|^{-3}\,\mathrm{d}\bar\mu=\,4\,\pi\,|\xi|^{-1}\,(|\xi|^2-1)^{-1}\\ &\circ\qquad\int_{S_1(\xi)}|x|^{-5}\,\mathrm{d}\bar\mu=\,\frac{4\,\pi}{3}\,|\xi|^{-1}\,(1+3\,|\xi|^2)\,(|\xi|^2-1)^{-3}
\end{align*} 
The following hold if $|\xi|\neq0,\,1$.
\begin{align*} 
&\circ\qquad\int_{S_1(\xi)}|x|^{-2}\,\mathrm{d}\bar\mu=\,2\,\pi\,|\xi|^{-1}\,\log\frac{1+|\xi|}{|1-|\xi||}\qquad\qquad\,\,\,\,\\ &\circ\qquad\int_{S_1(\xi)}|x|^{-4}\,\mathrm{d}\bar\mu=\,4\,\pi\,(1-|\xi|^2)^{-2}\\ &\circ\qquad\int_{S_1(\xi)}|x|^{-6}\,\mathrm{d}\bar\mu=\,4\,\pi\,(1+|\xi|^2)\,(1-|\xi|^2)^{-4}
\end{align*} 
\end{lem} 
\begin{lem}
	Let $\xi\in\mathbb{R}^3$. The following identities hold on ${{S}_{1}(\xi)}$.
	\label{inner product lemma}
	\begin{align*} 
&\circ\qquad	2\,\bar g(x,\bar\nu)=\,|x|^2+1-|\xi|^2\\[-2pt]
&\circ\qquad	2\,\bar g(x,\xi)=\,|x|^2+|\xi|^2-1\\[-2pt]
&\circ\qquad	2\,\bar g(\xi,\bar\nu)=\,|x|^2-1-|\xi|^2
	\end{align*}
\end{lem}

\section{Geometric expansions for perturbations of the Euclidean metric}
In this section, we collect some  expansions that relate geometric quantities computed with respect to different background metrics. \\ \indent We assume that $g$ is a Riemannian metric on $\mathbb{R}^3$ such that, as $x\to\infty$,
$$
g=\left(1+|x|^{-1}\right)^4\,\bar g+\sigma \qquad\text{where}\qquad \partial_J\sigma=O(|x|^{-2-|J|})
$$  for every multi-index $J$ with $|J|\leq 4$. We denote by $\tilde g=(1+|x|^{-1})^4\,\bar g$ the  Schwarzschild metric of mass $m=2$. We use a bar for geometric quantities pertaining to $\bar g$ and a tilde for quantities pertaining to $\tilde g$.  \\ \indent 
Recall that $e_1,\,e_2,\,e_3$ denotes the  standard basis of $\mathbb{R}^3$.
\begin{lem} \label{covariant derivative}
		Let $\xi\in\mathbb{R}^3$ and $i,\,j\in\{1,\,2,\,3\}$. There holds
	\begin{align*}  
	\tilde D_{e_i}\xi=2\,(1+|x|^{-1})^{-1}\,|x|^{-3}\,\left(\bar g(\xi,e_i)\,x-\bar g(\xi,x)\,e_i-\bar g(e_i,x)\,\xi\right).
	\end{align*} 
Moreover, as $x\to\infty$, \begin{align*} 
D_{e_i}\xi-\tilde D_{e_i}\xi=\,&O(|x|^{-3}),\\
D^2_{e_i,e_j}\xi-\tilde D^2_{e_i,e_j}\xi=\,&\frac12\,\sum_{k=1}^3\bigg[(\bar D_{e_i,\xi}^2\sigma)(e_j,e_k)+(\bar D_{e_i,e_j}^2\sigma)(\xi,e_k) -(\bar D_{e_i,e_k}^2\sigma)(\xi,e_j)\bigg]e_k+O(|x|^{-5}).
\end{align*} 
\end{lem}
\begin{lem}[{\cite[Lemma 37]{acws2}}] There holds
	\begin{align*}
\tilde{\operatorname{Ric}}(e_i,e_j)=2\,(1+|x|^{-1})^{-2}\,|x|^{-3}\left[\bar g(e_i,e_j)-3\,|x|^{-2}\,\bar g(e_i,x)\,\bar g(e_j,x)\right].
	\end{align*}
	Moreover, as $x\to\infty$,
	\label{ambient expansions} 	\begin{align*} 
&\operatorname{Ric}(e_i,e_j)-\tilde{\operatorname{Ric}}(e_i,e_j)\\&\qquad=\,\frac12\,\sum_{k=1}^3\bigg[(\bar D^2_{e_k,e_i}\sigma)(e_k,e_j)+(\bar D^2_{e_k,e_j}\sigma)(e_k,e_i)  -(\bar D^2_{e_k,e_k}\sigma)(e_i,e_j)-(\bar D^2_{e_i,e_j}\sigma)(e_k,e_k)\bigg]
\\&\qquad \qquad+O(|x|^{-5})
	\end{align*}
	and
	$$
	R=\sum_{i,\,j=1}^3\bigg[(\bar D^2_{e_i,e_j}\sigma)(e_i,e_j)-(\bar D^2_{e_i,e_i}\sigma)(e_j,e_j)\bigg]+O(|x|^{-5}).
	$$
\end{lem}
\begin{lem} Let $\xi\in\mathbb{R}^3$. 
	There holds, as $x\to\infty$, \label{second a expansion}
	\begin{equation*}
	\begin{aligned} 
	 \bar g(\tilde D_{e_1,e_1}^2\xi+\tilde D_{e_2,e_2}^2\xi,e_3)
 	=\tilde{\operatorname{Ric}}(\xi,e_3)+4\,|x|^{-4}\,\bar g(\xi,e_3)-8\,|x|^{-6}\,\bar g(\xi,x)\,\bar g(e_3,x)		+O(|x|^{-5}).		\end{aligned}
	\end{equation*}
\end{lem}
\begin{proof} 
	This follows from Lemma \ref{covariant derivative} and Lemma \ref{ambient expansions}.
\end{proof}

\begin{lem}
	Let $\{\Sigma_i\}_{i=1}^\infty$ be a sequence of surfaces $\Sigma_i\subset \mathbb{R}^3$ such that   
$
	\lim_{i\to\infty}\rho(\Sigma_i)=\infty.
$
The following expansions hold. \label{basic expansions} 	 
		\begin{equation*}
	\begin{aligned}  
&\circ\qquad	\tilde \nu=\,(1+|x|^{-1})^{-2}\,\bar\nu\\
&\circ\qquad	\tilde H=\,(1+|x|^{-1})^{-2}\,\bar H-4\,(1+|x|^{-1})^{-3}\,|x|^{-3}\,\bar g(x,\bar\nu) \qquad\qquad\,\,\,\,\\
&\circ\qquad	{\htildecirc}=\,(1+|x|^{-1})^{-2}\,\hbarcirc\\
&\circ\qquad \tilde\nabla \tilde h=\bar \nabla\bar h+O(|x|^{-3})+O(|x|^{-2}\,|\bar h|)+O(|x|^{-1}\,|\bar\nabla \bar h|)\\
&\circ\qquad	\mathrm{d}\tilde \mu=\,(1+|x|^{-1})^{4}\,\mathrm{d}\bar\mu
	\end{aligned} 
	\end{equation*}
	\begin{equation*}
\begin{aligned} 
&\circ\qquad\nu=\,\tilde\nu+O(|x|^{-2})\qquad\qquad\qquad\qquad\qquad\qquad\qquad\qquad\qquad\quad\\
&\circ\qquad H=\,\tilde H+O(|x|^{-3})+O(|x|^{-2}\,|\bar h|)\\
&\circ\qquad \hcirc=\,\htildecirc+O(|x|^{-3}) +O(|x|^{-2}\,|\bar h|)\\
&\circ\qquad \nabla h=\tilde \nabla\tilde h+O(|x|^{-4})+O(|x|^{-3}\,|\bar h|)+O(|x|^{-2}\,|\bar\nabla \bar h|)\\
&\circ\qquad \mathrm{d}\mu=\,[1+O(|x|^{-2})]\,\mathrm{d}\tilde \mu
	\end{aligned}
	\end{equation*} 
Moreover, if $\{u_i\}_{i=1}^\infty$ is a sequence of functions $u_i\in C^\infty(\Sigma_i)$, then
\begin{align*}
&\circ \qquad \tilde \nabla u_i=(1+|x|^{-1})^{-2}\,\bar \nabla  u_i,\\
&\circ \qquad \tilde \Delta u_i=(1+|x|^{-1})^{-4}\,\bar \Delta u_i,\qquad\qquad\qquad\qquad\qquad\,\,\,\,\,\qquad\qquad\quad
\end{align*} 
and
\begin{align*}
&\qquad\circ \qquad \nabla u_i=\tilde \nabla u_i+O(|x|^{-2}\,|\bar\nabla u_i|),\\
&\qquad\circ \qquad \Delta u_i=\tilde \Delta u_i+O(|x|^{-2}\,|\bar\nabla ^2 u_i|)+O(|x|^{-3}\,|\bar\nabla u_i|)+O(|x|^{-2}\,|\bar h|\,|\bar\nabla u_i|).
\end{align*}
\end{lem}
\section{Geometric expansions for graphs over Euclidean spheres}
In this section, we collect some geometric identities for graphs over Euclidean spheres. \\ \indent  
Let  $\xi\in\mathbb{R}^3$, $\lambda>0$, and $u\in C^\infty(S_\lambda(\lambda\,\xi))$. Recall from \eqref{Sigma def} that,  $\Sigma_{\xi,\lambda}(u)$ denotes the Euclidean graph of $u$ over $S_\lambda(\lambda\,\xi)$. 
\begin{lem} \label{graphical geometric components}
	The following identities hold.
	\begin{align*} 
	&\circ \qquad 	\bar g|_{\Sigma_{\xi,\lambda}(u)}=(1+\lambda^{-1}\,u)^2\,\bar g|_{S_{\lambda}(\lambda\,\xi)}+du\otimes du
	\\&\circ\qquad  \bar g|_{\Sigma}^{-1}=(1+\lambda^{-1}\,u)^{-2}\,\left[\bar g|^{-1}_{S_\lambda(\lambda\,\xi)}-((1+\lambda^{-1}\,u)^2+|\bar\nabla u|^2)^{-1}\,\bar\nabla u\otimes \bar\nabla u\right]
	\\&\circ \qquad \bar \nu(\Sigma_{\xi,\lambda}(u))=((1+\lambda^{-1}\,u)^2+|\bar\nabla u|^2)^{-1/2}\,((1+\lambda^{-1}\,u)\,\bar\nu(S_{\lambda}(\lambda\xi))-\bar\nabla u)
	\\&\circ \qquad \bar h(\Sigma_{\xi,\lambda}(u))=((1+\lambda^{-1}\,u)^2+|\bar\nabla u|^2)^{-1/2}
	\\&\qquad\qquad\qquad\qquad\qquad\big(\lambda^{-1}\,(1+\lambda^{-1}\,u)^2\,\bar g|_{S_{\lambda}(\lambda\xi)}+2\,\lambda^{-1}\,du\otimes du 
	-(1+\lambda^{-1}\,u)\,\bar\nabla^2 u\big)
	\end{align*} 
\end{lem}
\begin{lem} \label{second derivative estimate}
Suppose that
	\begin{align} \label{condition} 
	\lambda^{-1}\,|u|+|\bar\nabla u|\leq 1.
	\end{align} 
	There holds,  for all $f\in C^\infty (S_\lambda(\lambda\,\xi))$,
	\begin{align*} 
	\bar	\Delta_{\Sigma_{\xi,\lambda}(u)}f=\,&(1-2\,\lambda^{-1}\,u)\,\bar \Delta_{S_\lambda(\lambda\,\xi)}f
	\\&\qquad +O(|\bar\nabla f|\,|\bar\nabla u|\,(\lambda^{-2}\,|u|+\lambda^{-1}\,|\bar\nabla u|+|\bar\nabla^2 u|))
 	+O(|\bar\nabla ^2f|\,(\lambda^{-2}\,u^2+|\bar\nabla u|^2)).
	\end{align*} 
	
\end{lem}

\section{The potential function}
In this section, we collect some facts about the potential function of the spatial Schwarzschild manifold. \\ \indent We assume that $g$ is a Riemannian metric on $\mathbb{R}^3$. We denote by $\tilde g=(1+|x|^{-1})^4\,\bar g$ the  Schwarzschild metric of mass $m=2$. We use a bar for geometric quantities pertaining to $\bar g$ and a tilde for quantities pertaining to $\tilde g$. \\  
\indent Recall from \cite[\S 2]{Corvino} that the potential function  $N:\mathbb{R}^3\setminus\{0\}\to\mathbb{R}$ of the spatial Schwarzschild manifold is given by
\begin{align} \label{potential function} 
N(x)=(1+|x|^{-1})^{-1}\,(1-|x|^{-1}).
\end{align}
Moreover, recall that $N$ satisfies the static metric equation
\begin{align}  \label{static} 
\tilde D^2N=N\,\tilde{\operatorname{Rc}}.
\end{align} 
\indent
Let $\Sigma\subset M$ be a closed, two-sided surface with  outward normal $\nu$, mean curvature $H$ with respect to $\nu$, second fundamental form $h$, and non-positive Laplace-Beltrami operator $\Delta$  such that
$$
\Delta H+(|\hcirc|^2+\operatorname{Rc}(\nu,\nu)+\kappa)\,H=0
$$
for some $\kappa\in\mathbb{R}$. 
\begin{lem}
	Let $F:\Sigma\to\mathbb{R}$ be given by $F=N^{-1}\,H(\Sigma)$ and suppose that $X\in \Gamma(T\Sigma)$. 
	There holds \label{support function lemma}
	\begin{equation} \label{first willmore} 
	\begin{aligned}  
	\Delta F=\,&-\big(|\hcirc|^2+\kappa+\operatorname{Rc}(\nu,\nu)-\tilde{\operatorname{Rc}}(\tilde\nu,\tilde\nu)+N^{-1}\,\Delta N-N^{-1}\,\tilde \Delta N\\&\qquad-\tilde g(\tilde \nu,\tilde DN)\,\tilde H\big)\,F-2\,N^{-1}\,g(\nabla F,\nabla N).
	\end{aligned} 
	\end{equation} 
\end{lem}
\begin{proof}
	Note that
	$$
	\Delta F=-N^{-1}\,F\,\Delta N+2\,N^{-2}\,F\,g(\nabla N,\nabla N)-2\,N^{-2}\,g(\nabla H,\nabla N)+N^{-1}\,\Delta H.
	$$
	Using \eqref{static} and that $\tilde R=0$, we obtain
	$$
	\tilde	\Delta N=-N^{-1}\,\tilde{\operatorname{Rc}}(\tilde\nu,\tilde\nu)-\tilde g(\tilde \nu,\tilde D N)\,\tilde H.
	$$
	Clearly,
	$$
	\nabla H=N\,\nabla F+F\,\nabla N.
	$$
The assertion follows from these identities.
\end{proof}
\end{appendices}

\end{document}